\documentclass[twoside]{article}

\usepackage{amsmath,amsfonts,amssymb,amsthm,tikz,
}

\usepackage[a4paper]{geometry}

\usepackage{a4wide}

\newtheorem{thm}{Theorem}[section]
\newtheorem{specialthm}{Theorem}
\newtheorem{cor}[thm]{Corollary}
\newtheorem{lem}[thm]{Lemma}
\newtheorem{prop}[thm]{Proposition}

\theoremstyle{remark}
\newtheorem{rem}[thm]{Remark}

\theoremstyle{definition}
\newtheorem{definition}[thm]{Definition}
\newtheorem{example}[thm]{Example}

\newcommand{\Cech}{\v{C}ech{} }

\usepackage{color}\newcounter{marginnotes}
\let\oldmarginpar\marginpar
\renewcommand{\marginpar}[1]{\addtocounter{marginnotes}{1}
\oldmarginpar[\raggedleft\footnotesize #1]{\raggedright\footnotesize #1}}
\newcommand{\countnotes}{\ifnum\value{marginnotes}=0 {}
\else
\begin{center}
\color{red}\bf\Huge
\ifnum\value{marginnotes}=1 {There is 1 margin note in total.}
\else {There are \arabic{marginnotes} margin notes in total.}
\fi
\end{center}
\fi}

\usepackage{fancyhdr}

\begin{document}
	\author{Dan Rust \\
	\tt{dr151@le.ac.uk}
	}
	\title{An uncountable set of tiling spaces with distinct cohomology}
	\date{}
	\maketitle
		\begin{abstract}
		We generalise the notion of a Barge-Diamond complex, in the one-dimensional case, to a mixed system of tiling substitutions. This gives a way of describing the associated tiling space as an inverse limit of Barge-Diamond complexes. We give an effective method for calculating the \Cech cohomology of the tiling space via an exact sequence relating the associated sequence of substitution matrices and certain subcomplexes appearing in the approximants. As an application, we show that there exists a system of three substitutions on two letters which exhibit an uncountable collection of minimal tiling spaces with distinct isomorphism classes of \Cech cohomology.
		\end{abstract}
		
	\setcounter{section}{0}
	\section*{Introduction}
		The dynamical, topological and combinatorial properties of sequences derived from \emph{$S$-adic systems} have been well studied for many years (See \cite{DurandLeroyRichomme} and the references within). They have been classically defined as those sequences which appear as a limit word of a sequence of substitutions $s=\lim_{n \to \infty}((\phi_1 \circ \cdots \circ \phi_n)(a))$ for some letter in a finite alphabet $a$, and the substitutions $\phi_i$ belonging to some finite family of substitutions $S$. In \cite{GahlerMaloney}, G\"{a}hler and Maloney extended this class of sequences to an analogous class of tilings of $\mathbb{R}^n$ which generalise the now well-studied \emph{substitution tilings} -- these are referred to as \emph{mixed substitution} or \emph{multisubstitution} tiling systems. It is of general interest to be able to calculate the \Cech cohomology and other invariants of moduli spaces of aperiodic tilings, or \emph{tiling spaces} (See \cite{BargeDiamond, BargeDiamondHuntonSadun, ClarkHunton, ForrestHuntonKellendonk, FrankSadun, GahlerHuntonKellendonk, Sadun1, Walton}).
		
		Using techniques developed by authors such as Kellendonk \cite{Kellendonk}, Anderson \& Putnam \cite{AndersonPutnam}, G\"{a}hler \cite{Gahler}, and Barge \& Sadun \cite{BargeSadun}, a general method for calculating the \Cech cohomology of tiling spaces associated to mixed substitution systems was developed by G\"{a}hler and Maloney. In particular, they showed with an example that the topology of the associated tiling spaces of mixed substitution systems can be dependent on the order in which the substitutions are applied, and not just on the family of substitutions being considered.
		
		This example provided a set of $1$-dimensional mixed substitution tiling spaces over a fixed collection of two substitutions, some of whose first \Cech cohomology groups differ, and in fact have ranks varying depending on the choice of the sequence in which the substitutions are applied. This certainly hints that the family of mixed substitution tiling spaces has a richer structure than the classical case of tiling spaces associated to singular substitutions.
		
		One line of inquiry could ask how much more wild this class of spaces can appear to be. We already have some partial results. For instance it is well known that the Sturmian sequences can all be generated as limits of a system of two substitutions on two letters. It is also known that the tiling spaces associated to the Sturmian sequences $s_{\alpha}$ and $s_{\beta}$ are homeomorphic if and only if the generating slopes $\alpha$ and $\beta$ have continued fraction representations whose tails agree after some finite number of shifts \cite{Fokkink,BargeWilliams}. In particular, this tells us that there are an uncountable number of distinct homeomorphism-types of mixed substitution tiling spaces -- in contrast with the case of singular substitutions. Unfortunately, these spaces cannot be distinguished by their \Cech cohomology which are all isomorphic to the direct sum of two copies of the integers $\mathbb{Z}^2$.
		
		Our main theorem is an improvement on this result. We show, using a little-known theorem of Goodearl and Rushing \cite{Goodearl}, that an uncountable family exists even when we weaken the equivalence criteria to homotopy-equivalence.
		\begin{specialthm}\label{MainTheorem}
		There exists a family of minimal mixed substitution tiling spaces exhibiting an uncountable collection of distinct isomorphism classes of first \Cech cohomology groups.
		\end{specialthm}
		This has the (perhaps surprising) consequence that there exist tiling spaces $\Omega$ for which the first \Cech cohomology group $\check{H}^1(\Omega)$ cannot be written in the form
		$$A \oplus \left(\mathbb{Z}\left[\frac{1}{n_1}\right] \oplus \cdots \oplus \mathbb{Z}\left[\frac{1}{n_k}\right]\right)$$
		for finitely generated abelian group $A$ and natural numbers $n_i$, $1\leq i \leq k$, because there are only countably many distinct isomorphism classes of such groups. Moreover, it appears these pathological cohomology groups are in some sense typical. Nevertheless, almost every currently determined cohomology group of a tiling space is of the above form.
		
		In order to prove Theorem \ref{MainTheorem}, we leverage a construction by Barge and Diamond \cite{BargeDiamond} of the so-called BD (Barge-Diamond) complex of a tiling substitution. The BD complex is a CW complex associated to a single tiling substitution, built from combinatorial data and with the property that a suitably chosen continuous map on this complex, induced by the substitution, has an inverse limit which is homeomorphic to the tiling space. A cellular map can then be defined which is homotopic to this induced map, and which also acts simplicially on a particular subcomplex of the BD complex and maps this subcomplex into itself. A relative cohomology approach can then be used to produce an exact sequence which allows for the relatively straight forward computation of the \Cech cohomology of the tiling space in terms of the cohomology of a simplicial complex, and the direct limit of the transpose of the associated substitution transition matrix. These constructions and results were later generalised to tilings of $\mathbb{R}^n$ for all positive dimensions by Barge, Diamond, Hunton and Sadun in \cite{BargeDiamondHuntonSadun}.
		
		The original construction of these complexes, and the induced substitution maps between them, was only developed for single substitutions and so it is necessary to generalise their work to the case of mixed substitution systems. We essentially mirror the theory developed by G\"{a}hler and Maloney \cite{GahlerMaloney} for the Anderson-Putnam complex, but instead for the BD complex, and only in dimension-one.
		
		In Section \ref{Section:Background}, we introduce notation and basic definitions relating to the $3$-adic numbers. We provide an overview of the Goodearl-Rushing result \cite{Goodearl}, and for completeness and convenience to the reader, a proof of this result. Briefly, the result shows that the set of direct limits over $\mathbb{Z}^2$ of arbitrary sequences of matrices of the form $(\begin{smallmatrix} 1 & i \\ 0 & 3\end{smallmatrix})$, $i\in\{0,1,2\}$, have countable isomorphism classes. Notation is also introduced which will be used in the eventual proof of Theorem \ref{MainTheorem}.

		In Section \ref{Section:MixedSubstitutions}, we introduce the definitions and main results relating to mixed substitution systems and prove some important properties of their associated tiling spaces.

		In Section \ref{Section:BargeDiamond}, we construct, for a mixed substitution system, a sequence of BD complexes and induced maps between them. We show that the inverse limit of this sequence of maps is homeomorphic to the associated tiling space. We construct a sequence of homotopic cellular maps which can be used to effectively compute the first \Cech cohomology of the tiling space via an exact sequence. It was remarked in \cite{Gahler} that it might be possible to construct a `universal' BD complex for a family of mixed substitutions which satisfy some suitable property called \emph{self-correction}. The universal BD complex associated to a self-correcting mixed substitution system should satisfy the property that at each stage of the inverse limit representation of the tiling space, the approximant is the same, without affecting the cohomology of the limit. We provide one possible candidate for this property and show that such a universal BD complex can be constructed which behaves well in this sense if the mixed substitution system is self-correcting according to this definition.
		
		In Section \ref{chacon}, we use the Goodearl-Rushing result to define a family of mixed substitution systems we call the \emph{mixed Chacon tilings},  with the properties necessary to prove the main result. We prove that these properties are satisfied using the results from Section \ref{Section:BargeDiamond}.
		
		\textbf{Acknowledgements.} The author would like to thank his past and present supervisors, John Hunton and Alex Clark, for their helpful guidance. The author is supported by an EPSRC PhD studentship.

	\section{Background}\label{Section:Background}
		\subsection{The 3-adic numbers}
		We briefly review the notation surrounding the $3$-adic numbers and $3$-adic integers, which will be used throughout the statement and proof of the Goodearl-Rushing result. Let $\mathbb{Q}$ be the set of rational numbers on which we place the $3$-\emph{adic metric}. The metric is given by first defining the \emph{$3$-adic absolute value} $v_3\colon\mathbb{Q}\to\mathbb{R}$.
		\begin{definition}For a rational number $x$, if $x$ is non-zero write it in the form $x=3^n\frac{a}{b}$ where $n,a,b$ are integers, with $a,b$ not divisible by $3$. The \emph{$3$-adic valuation} $v_3\colon\mathbb{Q} \to \mathbb{Q} \cup \{\infty\}$ is given on $x$ by
		$$v_3(x)=
		\begin{cases}
		\infty, & \mbox{if } x=0\\
		n, & \mbox{if } x\neq 0
		\end{cases}.$$
		
		The \emph{$3$-adic absolute value} $| \cdot |_3 \colon\mathbb{Q} \to\mathbb{R}$ on $x$ is then given by $$|x|_3=3^{-v_3(x)}$$ where $3^{-\infty}$ is defined to be $0$.

		Given a pair of rational numbers $x$ and $y$, the \emph{$3$-adic metric} $d_3\colon\mathbb{Q}\times\mathbb{Q}\to\mathbb{R}$ is given by $$d_3(x,y)=|x-y|_3$$ which is quickly verified to be a metric.
		\end{definition}
		We can consider the completion of $\mathbb{Q}$ under the $3$-adic metric, known as the \emph{field of $3$-adic numbers} and denoted $\mathbb{Q}_3$. The addition and multiplication operations on $\mathbb{Q}_3$ are inherited from $\mathbb{Q}$ using the property that these operations on $\mathbb{Q}$ are uniformly continuous with respect to $3$-adic metric, and so extend to the completion.		
		
		Let $\mathbb{Z}_3$ be the set of $3$-adic integers, seen as a subset of $\mathbb{Q}_3$, given by the completion of the integers $\mathbb{Z}$ under the $3$-adic metric. The set $\mathbb{Z}_3$ forms a subring of $\mathbb{Q}_3$. The elements of $\mathbb{Z}_3$ are represented by sequences of digits $\epsilon_k\in\{0,1,2\}$ where $\epsilon_k$ is the $k$th digit of the $3$-adic integer $\alpha$ with unique standard expansion $\alpha=\sum_{k=0}^\infty\epsilon_k 3^k$. We denote the \emph{$n$-partial expansion} of $\alpha$ by $\alpha_n=\sum_{k=0}^n \epsilon_k 3^k$.

		\subsection{Uncountability of Isomorphism Classes of Direct Limits}\label{goodearl}
		For the benefit of the reader, this section is a self-contained proof of the result that will be needed in section \ref{chacon}. The proof was originally given in \cite{Goodearl}. It is fairly technical, and the material may be safely skipped, with the exception of the statement of Theorem \ref{countable}, and the definitions of the groups $G_{\alpha}$ and the equivalence relation $\sim$ on $\mathbb{Z}_3$.
		\subsubsection{The Goodearl-Rushing Method}
		Let $B_i=\begin{pmatrix}1 & i\\ 0 & 3\end{pmatrix}$ for $i\in\{0,1,2\}$. We will be considering direct systems consisting of sequences of these matrices. Let us fix a $3$-adic integer $\alpha\in\mathbb{Z}_3$ which has digits $\epsilon_n$ for $n\geq 0$ and let $G_{\alpha}=\displaystyle\lim_{\longrightarrow}(B_{\epsilon_n})$, the direct limit of the matrices $B_{\epsilon_n}$ acting on the group $\mathbb{Z}^2$.
		
		Let $V$ be a $2$-dimensional vector space over $\mathbb{Q}$. As the matrices $B_i$ are invertible over $\mathbb{Q}$, we can rewrite the above direct limit as a sequence of inclusions of rank-$2$ free abelian subgroups of $V$ such that the direct limit is then a union of these subgroups. We wish to do this in such a way that we keep track of the generating elements of the subgroups.
		
		Let $\{w_0,z_0\}$ be a $\mathbb{Q}$-basis of $V$ and define $A_{\alpha,0}=\langle w_0, z_0 \rangle$ to be the free abelian subgroup of $V$ generated by these elements. Set
		\begin{equation}\label{generators1}
		\begin{array}{rcl}
		w_n & = & w_0 \\
		z_n & = & 3^{-n}(z_0-\alpha_{n-1}w_0)
		\end{array}
		\end{equation}
		and similarly define $A_{\alpha,n}=\langle w_n,z_n \rangle$. It can be shown using the fact that $\alpha_n-\alpha_{n-1}=3^{n+1}\epsilon_n$ that		
		\begin{equation}\label{generators2}
		\begin{array}{rcl}
		z_n & = & 3z_{n+1}+\epsilon_n w_{n+1}
		\end{array}
		\end{equation}
		and so in particular $A_{\alpha, n}\subset A_{\alpha, n+1}$. Let $i_n\colon A_{\alpha,n}\to A_{\alpha,n+1}$ be the inclusion map.
		
		For fixed generators $a,b\in\mathbb{Z}^2$ define the group isomorphism $g_n\colon\mathbb{Z}^2\to A_{\alpha,n}$ by $g_n(ka + lb)=k w_n + l z_n$. It is easy to see using the relation in (\ref{generators2}) that $g_{n+1}\circ B_{\epsilon_n}=i_n\circ g_n$, meaning the diagram
		$$\begin{tikzpicture}[node distance=2cm, auto]
  		\node (00) {$A_{\alpha,0}$};
  		\node (10) [above of=00] {$\mathbb{Z}^2$};
  		\node (01) [right of=00] {$A_{\alpha,1}$};
  		\node (11) [right of=10] {$\mathbb{Z}^2$};
  		\node (02) [right of=01] {$A_{\alpha,2}$};
  		\node (12) [right of=11] {$\mathbb{Z}^2$};
  		\node (03) [right of=02] {$\cdots$};
  		\node (13) [right of=12] {$\cdots$};
  		\draw[->] (10) to node [swap] {$g_0$} (00);
  		\draw[->] (11) to node [swap] {$g_1$} (01);
  		\draw[->] (12) to node [swap] {$g_2$} (02);
  		\draw[->] (00) to node [swap] {$i_0$} (01);
  		\draw[->] (01) to node [swap] {$i_1$} (02);
  		\draw[->] (02) to node [swap] {$i_2$} (03);
  		\draw[->] (10) to node {$B_{\epsilon_0}$} (11);
  		\draw[->] (11) to node {$B_{\epsilon_1}$} (12);
  		\draw[->] (12) to node {$B_{\epsilon_2}$} (13);  
		\end{tikzpicture}$$
commutes, and so we may conclude that $$G_{\alpha}=\lim_{\longrightarrow}(B_{\epsilon_n}) \cong \lim_{\longrightarrow}(A_{\alpha,n},i_n)\cong\bigcup_{n\in\mathbb{N}}A_{\alpha,n}.$$
We will write $A_{\alpha}$ to denote $\bigcup A_{\alpha,n}$.
		\begin{prop}
		For any $\alpha\in\mathbb{Z}_3$, the group $A_{\alpha}$ is not finitely generated.
		\end{prop}
		\begin{proof}
		Consider the projection homomorphism $V\to\mathbb{Q}\colon w_0\mapsto 0, z_0\mapsto 1$. This projection restricted to the subgroup $\bigcup A_{\alpha,n}$ has image equal to the additive group of triadic integers $\mathbb{Z}[\frac{1}{3}]=\{a3^{-n}\mid a\in\mathbb{Z},n\in\mathbb{N}\}$ which is not finitely generated. It follows that $A_{\alpha}$ is infinitely generated.
		\end{proof}
		\subsubsection{Determining Isomorphisms}
		We now ask the question, given $\alpha=\ldots\epsilon_2\epsilon_1\epsilon_0$ and $\alpha'=\ldots\epsilon'_2\epsilon'_1\epsilon'_0$, when precisely is $A_{\alpha}$ isomorphic to $A_{\alpha'}$?
		\begin{thm}[Goodearl-Rushing]\label{countable}
		Let $\sim$ be the equivalence relation on $\mathbb{Z}_3$ given by $\alpha \sim \alpha'$ if and only if $G_{\alpha} \cong G_{\alpha'}$. The equivalence classes of $\sim$ are all countable.
		\end{thm}
		\begin{proof}
		Let us suppose that $G_\alpha\cong G_{\alpha'}$ for a particular pair $\alpha,\alpha'\in\mathbb{Z}_3$, then $A_{\alpha}\cong A_{\alpha'}$, and let $\varphi\colon A_{\alpha}\to A_{\alpha'}$ be a group isomorphism. We note that $A_{\alpha} \otimes_{\mathbb{Z}} \mathbb{Q}\cong A_{\alpha'} \otimes_{\mathbb{Z}} \mathbb{Q}\cong V$ and so in particular $\varphi$ extends uniquely to an automorphism $\tilde{\varphi}\colon V\to V$ of vector spaces. Let us represent $\tilde{\varphi}$ by its associated ($\mathbb{Q}$-invertible) matrix $\dfrac{1}{t}\begin{pmatrix}r_w & r_z\\ s_w & s_z\end{pmatrix}$ with respect to the basis $\{w_0,z_0\}$ of $V$, for integers $r_w,r_z,s_w,s_z,t$.
		
		We see that
		\begin{eqnarray}
		\tilde{\varphi}(w_0) & = &\frac{r_w}{t}w_0+\frac{s_w}{t}z_0\\
		\tilde{\varphi}(z_0) & = & \frac{r_z}{t}w_0+\frac{s_z}{t}z_0
		\end{eqnarray}
		and using (\ref{generators1}) we get for all $n\geq 0$ that
		\begin{eqnarray}\label{compare1}
		\tilde{\varphi}(z_n) & = & 3^{-n} \left(\frac{r_z}{t} - \alpha_{n-1}\frac{r_w}{t}\right) w_0 + 3^{-n} \left(\frac{s_z}{t}-\alpha_{n-1}\frac{s_w}{t}\right)z_0.
		\end{eqnarray}
		
		Given that $\tilde{\varphi}(A_{\alpha})=A_{\alpha'}$, there must exist a natural number $t\geq 0$ such that $\tilde{\varphi}(w_0)\in A_{\alpha',t}$ and $\tilde{\varphi}(z_0)\in A_{\alpha',t}$. Let $k(n)$ be the minimal natural number such that $\tilde{\varphi}(w_n)$ and $\tilde{\varphi}(z_n)$ are elements of $A_{\alpha',k(n)}$. Recalling (\ref{generators2}) we can also see that $\tilde{\varphi}(w_n)$ and $\tilde{\varphi}(z_n)$ are in $A_{\alpha',k(n+1)}$ and so in particular $k(n)\leq k(n+1)$ and so $(k(i))_{i\geq 0}$ is a non-decreasing sequence.
		
		Supposing $(k(i))_{i\geq 0}$ was bounded by some natural number $k$, we would have that $\tilde{\varphi}(w_n)$ and $\tilde{\varphi}(z_n)$ are in $A_{\alpha',k}$ for all $n\geq 0$ and so $\tilde{\varphi}(A_{\alpha})$ is a subgroup of $A_{\alpha',k}$. But note, $A_{\alpha',k}$ is finitely generated, and as $\tilde{\varphi}$ is a bijection, this would mean a finitely generated abelian group contained an infinitely generated subgroup. This can clearly not be the case and so we conclude that $(k(i))_{i\geq 0}$ must be unbounded.
		
		As we know $\tilde{\varphi}(z_n)\in A_{\alpha',k(n)}$, we then have integers $a_n,b_n$ such that
		\begin{eqnarray}\label{compare2}
		\tilde{\varphi}(z_n)=a_nw'_{k(n)}+b_nz'_{k(n)}.
		\end{eqnarray}
		Suppose $b_n$ is divisible by $3$ and so $b_n=3c_n$ for some integer $c_n$, then we find by substituting for $w'_n$ and $z'_n$ using (\ref{generators2}) that
		\begin{equation*}
		\tilde{\varphi}(z_n)=a_nw'_{k(n)-1}+c_n(z'_{k(n)-1}-\epsilon'_{k(n)-1}w'_{k(n)-1})
		\end{equation*}
		which would imply that $\tilde{\varphi}(z_n)\in A_{\alpha',k(n)-1}$, contradicting the minimality of $k(n)$. It follows that $b_n$ is not divisible by $3$.
		
		Comparing (\ref{compare1}) and (\ref{compare2}) we find that		
		$$
		3^{-n} \left(\frac{r_z}{t} - \alpha_{n-1}\frac{r_w}{t}\right) w_0 + 3^{-n} \left(\frac{s_z}{t}-\alpha_{n-1}\frac{s_w}{t}\right)z_0 = a_nw'_{k(n)}+b_nz'_{k(n)}
		$$		
		which, after substituting for $w'_{k(n)}$ and $z'_{k(n)}$ using (\ref{generators1}), becomes		
		$$
		3^{-n} \left(\frac{r_z}{t} - \alpha_{n-1}\frac{r_w}{t}\right) w_0 + 3^{-n} \left(\frac{s_z}{t}-\alpha_{n-1}\frac{s_w}{t}\right)z_0 = a_n w_0 + b_n 3^{-k(n)} (z_0-\alpha'_{k(n)-1}w_0).
		$$
		Picking out $w_0$ and $z_0$ components gives us
		\begin{equation}
		\begin{array}{rcl}\label{algebra}
		3^{-n}(r_z-\alpha_{n-1}r_w)/t & = & 3^{-k(n)}(3^{k(n)}a_n-\alpha'_{k(n)-1}b_n)\\
		3^{-n}(s_z-\alpha_{n-1}s_w)/t & = & 3^{-k(n)}b_n .
		\end{array}
		\end{equation}
		Cross-multiplying the equations in (\ref{algebra}) gives
		$$
		(r_z-\alpha_{n-1}r_w)b_n=(s_z-\alpha_{n-1}s_w)(3^{k(n)}a_n-\alpha'_{k(n)-1}b_n).
		$$
		As it is not a multiple of $3$ we can divide through by $b_n$ to give
		\begin{equation}\label{convseq}
		r_z-\alpha_{n-1}r_w=(s_z-\alpha_{n-1}s_w)(3^{k(n)}(a_n/b_n)-\alpha'_{k(n)-1}).
		\end{equation}
		We can now take the limit of the sequence of equations given in (\ref{convseq}) as $n\to\infty$ in the $3$-adic metric on $\mathbb{Q}_3$. Without loss of generality suppose $a_n\neq 0$ and write $a_n=3^{q(n)}d_n$ for an integer $d_n$ not divisible by $3$ and $q(n)\geq 0$ some natural number. Then it is clear that $$|3^{k(n)}a_n/b_n|_3=|3^{k(n)+q(n)}d_n/b_n|_3=3^{-(k(n)+q(n))}$$ and so, because $k(n)$ is a strictly increasing sequence of natural numbers, as $n$ tends to $\infty$, this valuation must tend to $0$. So $$\displaystyle\lim_{n\to\infty} d_3(3^{k(n)}a_n/b_n,0) = 0.$$ It follows that $\displaystyle\lim_{n\to\infty} 3^{k(n)}a_n/b_n=0$.
		
		Taking the limit of (\ref{convseq}) then tells us that
		$$\begin{array}{rrcl}
		&r_z-\alpha r_w & = & (\alpha s_w-s_z)\alpha'\\
		\implies & \alpha' & = & \dfrac{r_z-\alpha r_w}{\alpha s_w-s_z}
		\end{array}		
		$$
		and so we may finally conclude that for any given $\alpha\in\mathbb{Z}_3$, there are at most a countable number of distinct $3$-adic integers $\alpha'\in\mathbb{Z}_3$ such that $G_{\alpha}\cong G_{\alpha'}$, given by varying $r_w,r_z,s_w,s_z$ over the integers in the above equation. It follows that the $\sim$-equivalence classes on $\mathbb{Z}_3$ are all countable.
		\end{proof}

	\section{Mixed Substitutions}\label{Section:MixedSubstitutions}
		\subsection{Definitions}
		We now shift our attention to one-dimensional mixed substitution tilings. We would like to be able to compute the \Cech cohomology of the inverse limits of certain $1$-dimensional CW complexes in the style of Barge-Diamond \cite{BargeDiamond}, under maps induced by sequences of symbolic substitutions. In order to do this, we need to develop the theory of mixed substitution tiling spaces in this setting. Much of this theory, with regard to the Anderson-Putnam complexes \cite{AndersonPutnam}, was originally formulated by G\"{a}hler and Maloney \cite{GahlerMaloney}, and our notation will be closely based on theirs.
		
		Let $\mathcal{A}=\{a^1, a^2, \ldots, a^l\}$ be a finite alphabet on $l$ symbols and for positive integer $n$, let $\mathcal{A}^n$ be the set of words of length $n$ using symbols from $\mathcal{A}$. Denote the union of these by $\mathcal{A}^*=\bigcup_{n\geq 1}\mathcal{A}^n$. This set $\mathcal{A}^*$ forms a free semigroup under concatenation of words.
		\begin{definition}
		A \emph{substitution} $\phi$ on $\mathcal{A}$ is a function $\phi \colon \mathcal{A} \to \mathcal{A}^*$. We can extend the substitution $\phi$ in a natural way to a semigroup homomorphism $\phi \colon \mathcal{A}^* \to \mathcal{A}^*$ given, for a word $w = w_1\ldots w_n \in \mathcal{A}^n$, by setting $\phi(w) = \phi(w_1) \ldots \phi(w_n)$.
		\end{definition}
		We may further extend the above definition of a substitution to bi-infinite sequences $\phi\colon\mathcal{A}^{\mathbb{Z}}\to\mathcal{A}^{\mathbb{Z}}$. For a bi-infinite sequence $\mathcal{T}\in\mathcal{A}^{\mathbb{Z}}$, with $\mathcal{T} = \ldots a_{-2} a_{-1} \cdot a_0 a_1 a_2 \ldots$ we set $$\phi(\mathcal{T}) = \ldots \phi(a_{-2}) \phi(a_{-1}) \cdot \phi(a_0) \phi(a_1) \phi(a_2) \ldots$$ with the dot $\cdot$ representing the separator of the $(-1)$st and $0$th component of the respective sequences.
		
		If $\mathcal{A}=\{a^1,a^2,\ldots,a^l\}$ is an alphabet with a substitution $\phi\colon\mathcal{A} \to \mathcal{A}^*$, then $\phi$ has an associated substitution matrix $M_{\phi}$ of dimension $l\times l$ given by setting $m_{ij}$, the $i,j$ entry of $M_{\phi}$, to be the number of times that the letter $a^i$ appears in the word $\phi(a^j)$.
		\begin{definition}
		A substitution $\phi \colon \mathcal{A} \to \mathcal{A}^*$ is called \emph{primitive} if there exists a positive natural number $p$ such that the matrix $M_{\phi}^p$ has strictly positive entries. Equivalently, if there exists a positive natural number $p$ such that for all $a,a'\in\mathcal{A}$ the letter $a'$ appears in the word $\phi^p(a)$.
		\end{definition}
		Let $F=\{\phi_0, \phi_1,\ldots, \phi_k\}$ be a finite set of substitutions on $\mathcal{A}$. Consider an infinite sequence $s=(s_0, s_1, s_2, \ldots)\in\{0,1,\ldots, k\}^{\mathbb{N}}$. For a fixed alphabet $\mathcal{A}$, we call a pair $(F,s)$ a \emph{mixed substitution system}. For natural numbers $k\leq l$, let $M_{s[k,l]}=M_{s_k} M_{s_{k+1}}\cdots M_{s_{l-1}} M_{s_l}$ be the associated substitution matrix of the substitution $\phi_{s[k,l]}=\phi_{s_k} \phi_{s_{k+1}} \cdots \phi_{s_{l-1}} \phi_{s_l}$.
		\begin{definition}
		The mixed substitution system $(F,s)$ is called \emph{weakly primitive} if for all natural numbers $n$, there exists a positive natural number $k$ such that the matrix $M_{s[n,n+k]}$ has strictly positive entries.
		
		The mixed substitution system $(F,s)$ is called \emph{strongly primitive} if there exists a positive natural number $k$ such that for all natural numbers $n$, the matrix $M_{s[n,n+k]}$ has strictly positive entries.
		\end{definition}
		\begin{rem}
		Strongly primitive substitution systems are referred to simply as \emph{primitive} in \cite{GahlerMaloney} and \emph{bounded primitive} in \cite{PachechoVilarinho}. There are still other conventions for this notation \cite{DurandLeroyRichomme}. We take the convention that whenever we say a system is primitive, we will mean weakly primitive in the sense defined above.
		\end{rem}
		To a mixed substitution system $(F,s)$ we will associate a topological space $\Omega_{F,s}$ called the \emph{continuous hull} or \emph{tiling space} of the mixed substitution system $(F,s)$. Let $\mathcal{A}^{\mathbb{Z}}$ be the space of all bi-infinite sequences of symbols in $\mathcal{A}$ with the usual metric.
		\begin{definition}
		Let $(F,s)$ be as above and let $\mathcal{T}\in\mathcal{A}^{\mathbb{Z}}$ be a bi-infinite sequence $$\ldots a_{-2} a_{-1} a_0 a_1 a_2\ldots$$ of symbols $a_i\in\mathcal{A}$. We say $\mathcal{T}$ is \emph{admitted} by $(F,s)$ if for all $t\in\mathbb{Z}$ and all $u\geq 0$, there exists a $k\geq 0$ and a letter $a\in\mathcal{A}$ such that the string $a_t a_{t+1} \ldots a_{t+u-1} a_{t+u}$ appears as a substring of the word $\phi_{s[0,k]}(a)$. We call the subspace of $\mathcal{A}^\mathbb{Z}$ containing all bi-infinite sequences admitted by $(F,s)$ the \emph{discrete hull} of $(F,s)$ and denote it by $\Sigma_{F,s}$.
		
		We say a word $w$ is admitted by $(F,s)$ if it appears as a subword of an $(F,s)$-admitted bi-infinite sequence. We denote the set of length-$n$ $(F,s)$-admitted words by $\mathcal{L}^n_{F,s}$. The set $\mathcal{L}_{F,s}=\bigcup \mathcal{L}^n_{F,s}$ is referred to as the \emph{language} of $(F,s)$.
		\end{definition}	
		\begin{rem}
		It is worth remarking that a pair of theorems of Durand provides a close link between the strength of primitivity of a mixed substitution system and the growth rate of its repetitivity function. We do not explain the terms in these statements but instead refer the reader to \cite{Durand1,Durand2}.
		\begin{thm}[Durand]
		A sequence is repetitive (equivalently its shift is uniformly recurrent) if and only if it is admitted by a weakly primitive mixed substitution system.
		\end{thm}
		\begin{thm}[Durand]\label{durand2}
		A sequence is linearly repetitive (equivalently its shift is linearly recurrent) if and only if it is admitted by a strongly primitive and proper mixed substitution system.
		\end{thm}
		\end{rem}	
		Let $\sigma\colon \mathcal{A}^{\mathbb{Z}} \to \mathcal{A}^{\mathbb{Z}}$ be the left shift map and note that $\sigma|_{\Sigma_{F,s}}$, the restriction of $\sigma$ to $\Sigma_{F,s}$, has image equal to $\Sigma_{F,s}$ and so is a homeomorphism of the discrete hull.
		\begin{definition}
		Let $\Omega_{F,s}$ be the suspension of the restricted left shift $\sigma|_{\Sigma_{F,s}}$, given by $$\Omega_{F,s}=(\Sigma_{F,s}\times [0,1])/{\sim}$$ where $(\mathcal{T},1)\sim(\sigma(\mathcal{T}),0)$. We call $\Omega_{F,s}$ the \emph{continuous hull} or \emph{tiling space} of the mixed substitution system $(F,s)$.
		\end{definition}
		For a bi-infinite sequence $\mathcal{T}=\ldots a_{-2} a_{-1} \cdot a_0 a_1 a_2\ldots \in \Sigma_{F,s}$ and real number $t\in[0,1)$, we can think of a point in $\Omega_{F,s}$ as being a tiling of $\mathbb{R}$ with unit tiles coloured by the symbols of $\mathcal{A}$. A point $(\mathcal{T},t)\in\Omega_{F,s}$ denotes the tiling with a unit $a_0$-tile at the origin, and where $t$ describes the point in the $a_0$-tile over which the origin lies. There is then a unit $a_{-1}$-tile and a unit $a_1$-tile to the left and right respectively of the $a_0$-tile, and so on.
		\begin{rem}
		We note that the G\"{a}hler-Maloney approach \cite{GahlerMaloney} retains geometric data associated to tilings in the form of tile lengths and expanding factors of substitutions -- this is especially important if one wants to generalise to tilings in arbitrary dimensions, or want to study the natural dynamical and ergodic properties of the tiling space. However, this means that it is more difficult to handle systems with substitutions whose expanding factors do not exist, as in the case when the individual substitutions are not primitive, or whose length vectors (given by the left Perron-Frobenius eigenvector) do not coincide. By only considering the combinatorial data in dimension-one, we are able to handle these cases rather easily and, as will be seen, only require primitivity of the mixed substitution system, and not the individual substitutions -- many of our examples fall into these cases, including the example used to prove Theorem \ref{MainTheorem} and the well-known pair of substitutions which generate the Sturmian sequences. We will look at a generalised example of the Sturmian substitutions in Example \ref{Ex:ArnouxRauzy}.	
		\end{rem}

		\subsection{Properties of the Tiling Space}
		\begin{prop}
		If $(F,s)$ is weakly primitive, then $\Omega_{F,s}$ is non-empty.
		\end{prop}
		\begin{proof}
		Without loss of generality, assume $|\mathcal{A}| > 1$. As $(F,s)$ is weakly primitive, suppose $k(n)$ is such that $M_{s[n,n+k(n)]}$ has strictly positive entries for $n\geq 0$. Let $n_0=n$, and $n_{i+1}=n_i+k(n_i)+1.$ This means that the matrix $$M_{s[n_0,n_3-1]}=M_{s[n_0,n_0+k(n_0)]}M_{s[n_1,n_1+k(n_1)]}M_{s[n_2,n_2+k(n_2)]}$$ will have all entries at least as large as $3$. From this we see that for some letter $a\in\mathcal{A}$ there must be a copy of the letter $a$ appearing in the interior\footnote{The interior of the word $a_1a_2\ldots a_{i-1}a_i$ is the word $a_2\ldots a_{i-1}$.} of the word $\phi_{s[n_0,n_3-1]}(a)$, since the leftmost and rightmost letters account for at most two $a$s. Let $n=0$. By considering the sequence of words $$\phi_{s[0,n_3-1]}(a), \phi_{s[0,n_6-1]}(a), \ldots, \phi_{s[0,n_{3i}-1]}(a),\ldots$$ we see that the $i$th word appears as a subword of the $(i+1)$st word. As $a$ appears in the interior of $\phi_{s[0,n_3-1]}(a)$, this sequence of words expands in both directions, where here we assign an `origin' to the interior element $a$. In the limit of this (not necessarily unique) sequence of increasing containments of words, we are left with a bi-infinite sequence which by construction is admitted by $(F,s)$.	It follows that $\Sigma_{F,s}$ is non-empty, and then so is $\Omega_{F,s}$.
		\end{proof}
		The next result is a consequence of the right to left implication of Theorem \ref{durand2} but we provide an elementary proof for completeness.
		\begin{prop}
		If $(F,s)$ is weakly primitive, then the translation action on the tiling space $\Omega_{F,s}$ is minimal.
		\end{prop}
		\begin{proof}
		Let $w$ be a word which is admitted by $(F,s)$ and let $n\geq 0$ and $a\in\mathcal{A}$ be such that $w$ appears as a substring of the word $\phi_{s[0,n]}(a)$. Let $\mathcal{T} \in \Sigma_{F,s}$ be a limit of applying the sequence of substitutions $(\phi_{s[0,i]})_{i\geq 0}$ to the letter $a \in \mathcal{A}$. We will show that there are bounded gaps between subsequent occurrences of the word $w$ in the bi-infinite sequence $\mathcal{T}$.
		
		As the letter $a$ appears in $\phi_{s[0,k(0)]}(b)$ for all $b \in \mathcal{A}$, the word $w$ must then appear as a subword of the words $\phi_{s[0,n+k(0)]}(b)$ for all $b \in \mathcal{A}$. Let $L$ be the maximum of the lengths of the words $\phi_{s[0,n+k(0)]}(b)$ over all $b\in\mathcal{A}$. The bi-infinite sequence $\mathcal{T} \in \Sigma_{F,s}$ can be decomposed into a concatenation of these words $\phi_{s[0,n+k(0)]}(b)$ and so it follows that the word $w$ appears in $\mathcal{T}$ with gap at most $2L$.
		
		It follows that for all bi-infinite sequences $\mathcal{T}' \in \Sigma_{F,s}$, and for all $\epsilon>0$, there exists a $k\geq 0$ such that $d(\sigma^k(\mathcal{T}), \mathcal{T}')< \epsilon$ and so $\mathcal{T}'$ belongs to the closure of the shift orbit of $\mathcal{T}$. So $(\Sigma_{F,s},\sigma)$ is a minimal dynamical system. This readily implies that the translation action inherited by $\Omega_{F,s}$ is minimal.
		\end{proof}
		In particular, for a weakly primitive system $(F,s)$ and for any bi-infinite sequence $\mathcal{T}\in\Sigma_{F,s}$, we have $\Sigma_{F,s}=\overline{\{\sigma^k(\mathcal{T})\mid k\in\mathbb{Z}\}}$ and $\Omega_{F,s}=\overline{\{(\sigma^{\lfloor t\rfloor}(\mathcal{T}),t-\lfloor t\rfloor)\mid t\in\mathbb{R}\}}$. Here, $\lfloor - \rfloor$ is the floor function.
		\begin{definition}
		Let $\mathcal{T} = \ldots a_{-2} a_{-1} \cdot a_0 a_1 a_2 \ldots$ be a bi-infinite sequence in $\Sigma_{F,\sigma^{i+1}(s)}$ and let $t\in[0,1)$, so that $(\mathcal{T},t)$ is an element of the tiling space $\Omega_{F,\sigma^{i+1}(s)}$. We define a map between tiling spaces which we call $\phi_{s_i}\colon \Omega_{F,\sigma^{i+1}(s)} \to \Omega_{F,\sigma^i(s)}$, given by $$\phi_{s_i}(\mathcal{T},t)=(\sigma^{\lfloor \tilde{t} \rfloor}(\phi_{s_i}(\mathcal{T})), \tilde{t}-\lfloor \tilde{t} \rfloor)$$ where $\tilde{t}=|\phi_{s_i}(a_0)|\cdot t$, and where $|\phi_{s_i}(a)|$ is the length of the substituted word $\phi_{s_i}(a)$. 
		\end{definition}
		This map is continuous. Intuitively, we take a unit tiling in $\Omega_{F,\sigma^{i+1}(s)}$ with a prescribed origin, partition each tile of type $a$ uniformly with respect to the substituted word $\phi_{s_i}(a)$ into tiles of length $\frac{1}{|\phi_{s_i}(a)|}$. Then expand each tile away from the origin so that each new tile is again of unit length, and with the origin lying proportionally above the tile it appears in after partitioning the original tiling.
		\begin{definition}
		A mixed substitution tiling system $(F,s)$ is said to be \emph{recognisable} if for every $i\geq 0$ the map $\phi_{s_i}\colon \Omega_{F,\sigma^{i+1}(s)} \to \Omega_{F,\sigma^i(s)}$ is injective.
		\end{definition}
		\begin{definition}
		A mixed substitution tiling system $(F,s)$ has the \emph{unique composition property} if for any $i\geq 0$ and $\mathcal{T} \in \Sigma_{F,\sigma^i(s)}$, there is a unique $\mathcal{T}' \in \Sigma_{F,\sigma^{i+1}(s)}$ such that $\phi_{s_i}(\mathcal{T}')=\mathcal{T}$. Equivalently, there is a unique way of partitioning the symbols in $\mathcal{T} = \ldots a_{-1} a_0 a_1 \ldots$ into words which are substituted letters $\ldots \phi_{s_i}(a'_{-1})\phi_{s_i}(a'_0)\phi_{s_i}(a'_1)\ldots=\mathcal{T}$.
		\end{definition}
		It is an exercise to show that $\phi_{s_i}$ is always surjective, and so one notes that recognisability of a mixed substitution system is equivalent to it having the unique composition property.
				

		\section{Inverse Limits and Cohomology}\label{Section:BargeDiamond}
		\subsection{The Barge-Diamond Complex for Mixed Substitutions}\label{BargeDiamond}
		In \cite{BargeDiamond} Barge and Diamond introduced a cell complex for $1$-dimensional substitution tilings which we refer to as the BD complex. This was later extended by Barge, Diamond, Hunton and Sadun in \cite{BargeDiamondHuntonSadun} to arbitrary dimensions and allowed for symmetry groups beyond translations. Intuitively, we think of their complex as being the Anderson-Putnam complex of a collared version of the tiling, but where we collar points instead of tiles. These collared points then retain transition information between tiles and so induce \emph{border forcing} -- a term originally coined by Kellendonk \cite{Kellendonk} and utilised by Anderson and Putnam \cite{AndersonPutnam} in their seminal paper. Under a suitable choice of map on the BD complex induced by the substitution, they produced an inverse system whose limit is homeomorphic to the relevant tiling space. The complex has the advantage of being more manageable than the Anderson-Putnam collared complex for the computation of cohomology groups, as well as giving conceptually insightful information about where the generators of cohomology are coming from with regard to the tiling.
		
		The appearance of an exact sequence coming from considering the relative cohomology groups of their complex, and a certain subcomplex of `vertex edges', allows for the cohomology of a substitution tiling space to be built from relatively easy to compute pieces -- most notably for us, one of these pieces is the direct limit of the transpose of the original substitution matrix. 
		
		In order to apply their technique to our setting, we need to extend their method to more general sequences of substitutions, much like G\"ahler and Maloney did in \cite{GahlerMaloney} for the Anderson-Putnam complex in their treatment of mixed substitution systems.
		\begin{definition}
		Let $\mathcal{A}$ be a finite alphabet and let $(F,s)$ be a primitive substitution system over $\mathcal{A}$. Let $\displaystyle{\epsilon=\min_{ a\in\mathcal{A}, \phi\in F} \left\{\frac{1}{2|\phi(a)|}\right\}}$ be a small positive real number. For $a\in\mathcal{A}$, let $$e_a=[\epsilon, 1-\epsilon]\times\{a\}$$ and for $ab\in\mathcal{L}^2_{F,s}$, let $$e_{ab}=[-\epsilon,\epsilon]\times\{ab\}.$$ The \emph{Barge-Diamond complex} for the mixed substitution system $(F,s)$ is denoted by $K_{F,s}$ and is defined to be
		$$K_{F,s} = \left( \bigcup_{a\in\mathcal{A}} e_a \cup \bigcup_{ab\in\mathcal{L}^2_{F,s}}e_{ab} \right)/{\sim}$$
		where for all $a,b,c\in\mathcal{A}$,
		$$(1-\epsilon,a) \sim(-\epsilon,ab) \sim (-\epsilon,ac) \quad \mbox{and} \quad (\epsilon,a) \sim (\epsilon,ba) \sim (\epsilon, ca).$$
		We also define the \emph{subcomplex of vertex edges} $S_{F,s}$ of $K_{F,s}$ by
		$$S_{F,s}=\bigcup_{ab\in\mathcal{L}^2_{F,s}}e_{ab}/{\sim}.$$
		The other edges $e_a$ in $K_{F,s}$ are called \emph{tile edges}.
		\end{definition}		
		\begin{rem}
		This is only a slight modification of the usual Barge-Diamond complex for a unit-length substitution tiling, mainly in the choice of $\epsilon$ and lengths of edges -- we necessarily lose geometric information about tile lengths because we do not necessarily have a compatible Perron-Frobenius eigenvalue, as our transition matrices may not necessarily be individually primitive or have coinciding spectra.
		\end{rem}
		From the construction, a vertex edge $e_{ab}$ is only included in $K_{F,s}$ if the two-letter word $ab$ is admitted by $(F,s)$. This means that the BD complex for a system $(F,s)$ will be dependent on $s$. In particular, the sequence of complexes $(K_{F,\sigma^i(s)})_{i\geq 0}$ may not be constant, as would be the case in the classical setting of a single substitution where $F=\{\phi\}$.
		
		We also remark that the subcomplex $S_{F,\sigma^i(s)}$ need not be connected. Such an example was given by Barge and Diamond in \cite{BargeDiamond} where $\mathcal{A}=\{a,b,c,d\}$ is an alphabet on four letters and $F=\{\phi\}$ is a single substitution given by
		$$
		\phi:\left\{
		\begin{array}{cc}
			\begin{array}{rcl}
			a & \mapsto & abcda \\
			b & \mapsto & ab
			\end{array}
			&
			\begin{array}{rcl}
			c & \mapsto & cdbc \\
			d & \mapsto & db
			\end{array}
		\end{array}
			\right.
		$$
		which admits pairs $\mathcal{L}^2_{\phi}=\{aa, ab, ba, bc, da, db\} \cup \{cd\}$ where the vertex edge $e_{cd}$ is disjoint from the rest of the vertex edges appearing in $S_{\phi}$.
		\begin{example}
		As an example for where the BD complexes for $(F,\sigma^i(s))$ and $(F,\sigma^j(s))$ can differ, consider the set of substitutions $F=\{\phi_0, \phi_1\}$ on the alphabet $\mathcal{A}=\{a,b\}$ given by
		$$
		\begin{array}{cc}
			\phi_0:\left\{
			\begin{array}{rcl}
			a & \mapsto & b \\
			b & \mapsto & ba
			\end{array}
			\right., 
			&
			
			\phi_1:\left\{
			\begin{array}{rcl}
			a & \mapsto & ab \\
			b & \mapsto & ba 
			\end{array}
			\right.
		\end{array}.
		$$
		Here, $\phi_0$ is the Fibonacci substitution and $\phi_1$ is the Thue-Morse substitution. Let $s$ be the sequence $s=(1,0,0,0\ldots)$, so $s_0=1$ and $s_i=0$ for all $i\geq 1$.
		
		It can be easily checked that $\mathcal{L}^2_{F,s}=\{aa,ab,ba,bb\}$ and $\mathcal{L}^2_{F,\sigma(s)}=\{ab,ba,bb\}$. So $K_{F,s}$ and $K_{F,\sigma(s)}$ have a different number of vertex edges. After that, $K_{F\sigma^i(s)}$ will be equal to $K_{F,\sigma(s)}$ for all positive $i$. See Figure \ref{fig:M1} for the associated BD complexes. We have labelled the tile edges $e_x$ and the vertex edges $e_{xy}$ by their indices for better readability.
		\end{example}
		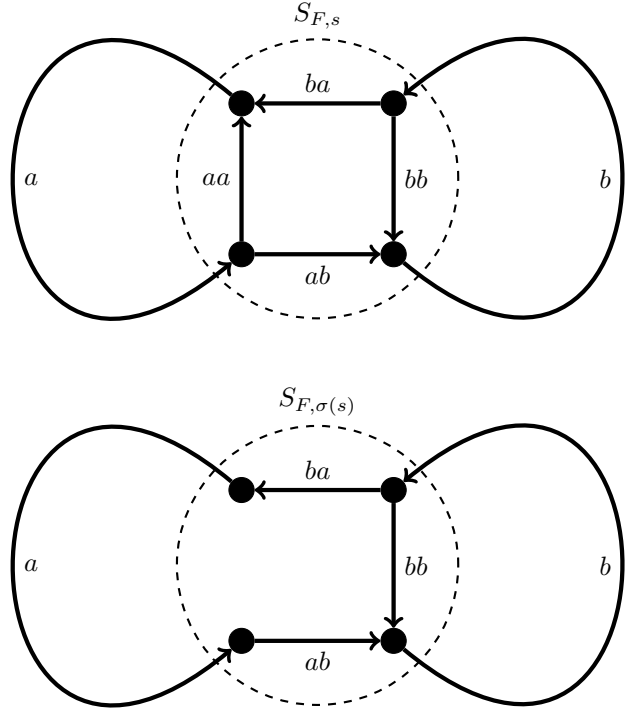
\begin{figure}[h]
		$$
		\begin{array}{c}\begin{tikzpicture}[node distance=2cm, auto]
		\clip (-5,-2.5) rectangle (5,2.5) ;
		\node [draw, circle, minimum size=3.7cm, thick, dashed] (xx) [label=$S_{F,s}$]{};
  		\node [draw, fill, circle, minimum size=.15cm] (00) [below left of=xx, node distance=1.4142cm] {};
  		\node [draw, fill, circle, minimum size=.15cm] (10) [above of=00] {};
  		\node [draw, fill, circle, minimum size=.15cm] (01) [right of=00] {};
  		\node [draw, fill, circle, minimum size=.15cm, use as bounding box] (11) [right of=10] {};
  		
  		
  		\draw [->,ultra thick] (00) to node {$aa$} (10);
  		\draw [->,ultra thick] (11) to node {$bb$} (01);
  		\draw [->,ultra thick] (00) to node [swap] {$ab$} (01);
  		\draw [->,ultra thick] (11) to node [swap] {$ba$} (10);

  		\path (10) edge[ out=140, in=220
                , looseness=0.1, loop
                , distance=5cm, ->
                , ultra thick]
            node {$a$} (00);
            
        \path (01) edge[ out=320, in=40
                , looseness=0.1, loop
                , distance=5cm, ->
                , ultra thick]
            node {$b$} (11);
		\end{tikzpicture}
\\		
		\begin{tikzpicture}[node distance=2cm, auto]
		\clip (-5,-2.5) rectangle (5,2.5) ;
		\node [draw, circle, minimum size=3.7cm, thick, dashed] (xx) [label=$S_{F,\sigma(s)}$]{};
  		\node [draw, fill, circle, minimum size=.15cm] (00) [below left of=xx, node distance=1.4142cm] {};
  		\node [draw, fill, circle, minimum size=.15cm] (10) [above of=00] {};
  		\node [draw, fill, circle, minimum size=.15cm] (01) [right of=00] {};
  		\node [draw, fill, circle, minimum size=.15cm, use as bounding box] (11) [right of=10] {};
  		
  		
  		\draw [->,ultra thick] (11) to node {$bb$} (01);
  		\draw [->,ultra thick] (00) to node [swap] {$ab$} (01);
  		\draw [->,ultra thick] (11) to node [swap] {$ba$} (10);

  		\path (10) edge[ out=140, in=220
                , looseness=0.1, loop
                , distance=5cm, ->
                , ultra thick]
            node {$a$} (00);
            
        \path (01) edge[ out=320, in=40
                , looseness=0.1, loop
                , distance=5cm, ->
                , ultra thick]
            node {$b$} (11);
		\end{tikzpicture}
		\end{array}
		$$
		\caption{The Barge-Diamond complexes $K_{F,s}$ and $K_{F,\sigma(s)}$ with their subcomplexes of vertex cells $S_{F,s}$ and $S_{F,\sigma(s)}$ circled accordingly.}\label{fig:M1}
		\end{figure}
		\begin{definition}Let $(F,s)$ be a mixed substitution system and let $(\mathcal{T},t)$ be a tiling in the tiling space $\Omega_{F,\sigma^i(s)}$ with $t\in [0,1)$. Suppose $\mathcal{T}=\ldots a_{-1}\cdot a_0 a_1\ldots$. We define a surjective continuous map $p_i\colon\Omega_{F,\sigma^i(s)}\to K_{F,\sigma^i(s)}$ by
		$$p_i(\mathcal{T},t)=
		\begin{cases}
		(t,a_{-1}a_0) & \mbox{if } t\in[0,\epsilon]\\
		(t,a_0) & \mbox{if } t\in[\epsilon, 1-\epsilon]\\
		(t-1, a_0a_1) & \mbox{if } t\in[1-\epsilon,1)
		\end{cases}$$
		\end{definition}
		There is a unique continuous map $f_i\colon K_{F,\sigma^{i+1}(s)} \to K_{F,\sigma^i(s)}$ such that $f_i\circ p_{i+1}=p_i\circ\phi_{s_i}$ induced in the obvious way by the substitution on edges.
		\begin{rem}The maps $f_i$ are not cellular with respect to the tile and vertex edges of the BD complexes because the edges are, in general, expanded as they are substituted; for instance vertex edges $e_{ab}$ are not just mapped onto some other vertex edge $e_{cd}$, but also overlap into the adjacent tile edges $e_c$ and $e_d$. We account for this later by introducing a cellular map $g_i$ which is homotopic to $f_i$.
		\end{rem}
		To reduce notation, we assume primitivity and recognisability always hold from this point. Also for notational convenience, let $K_i = K_{F,\sigma^i(s)}$, $S_i = S_{F,\sigma^i(s)}$, $\Omega_i = \Omega_{F,\sigma^i(s)}$, and set $\Omega = \Omega_0 = \Omega_{F,s}$.
		\begin{thm}\label{homeomorphism}
		For a primitive, recognisable mixed substitution system $(F,s)$, there is a homeomorphism $$\Omega \cong \lim_{\longleftarrow}(K_i,f_i)$$ between the mixed substitution tiling space and the inverse limit of the associated inverse system of induced substitution maps on the Barge-Diamond complexes.
		\end{thm}
		The proof is essentially identical to the one given by Barge and Diamond in the case of a single substitution.
		\begin{proof}
		From the definition of the maps $f_i$, we have commuting diagrams
		$$\begin{tikzpicture}[node distance=2.5cm, auto]
  		\node (00) {$K_i$};
  		\node (10) [node distance=2cm, above of=00] {$\Omega_i$};
  		\node (01) [right of=00] {$K_{i+1}$};
  		\node (11) [right of=10] {$\Omega_{i+1}$};

  		\draw[->] (10) to node [swap] {$p_i$} (00);
  		\draw[->] (11) to node [swap] {$p_{i+1}$} (01);
  		\draw[->] (01) to node [swap] {$f_i$} (00);
  		\draw[->] (11) to node [swap] {$\phi_{s_i}$} (10);
		\end{tikzpicture}$$
		for each $i\geq 0$. We note that that the maps $\phi_{s_i}$ are homeomorphisms by recognisability and compactness. These commutative diagrams induce a map $p\colon \Omega \to \displaystyle{\lim_{\longleftarrow}(K_i,f_i)}$ given by
		$$p(x) = (p_0(x), p_1(\phi_{s_0}^{-1}(x)), p_2(\phi_{s_1}^{-1}(\phi_{s_0}^{-1}(x))),\ldots, p_{i+1}(\phi_{s[0,i]}^{-1}(x)),\ldots).$$
		As each $p_i$ is surjective, so then is $p$.
		
		Let $V=\{(\epsilon,a)\}_{a\in\mathcal{A}} \cup \{(1-\epsilon,a)\}_{a\in\mathcal{A}}$ and choose a point $y=(y_0, y_1, \ldots) \in \displaystyle{\lim_{\longleftarrow} f_i }$. If there exists an $n\geq 0$ such that $y_n\in V$, then since $\epsilon$ was chosen small enough, $y_{n-1}$ will be in $e_a\setminus V$ for some symbol $a\in\mathcal{A}$. So, if $(\mathcal{T},t)$ is in $p^{-1}(y)$, then the $0$th tile of $\phi_{s[0,n-1]}^{-1}(\mathcal{T},t)$ is determined, and also the placement of the origin in the interior of this tile. So a patch of length $|\phi_{s[0,n]}(a)|$ around the origin in $(\mathcal{T},t)$ is also determined, and by primitivity if $y_n\in V$ for arbitrarily large $n$, then the length of this determined patch increases without bound. It follows that the entire tiling $(\mathcal{T},t)$ is determined.
		
		If $y_n$ is not in $V$ for arbitrarily large $n$ then there is some $N\geq 0$ such that for all $n\geq N$, either $y_n\in e_{ab}$ for a pair $a,b\in\mathcal{A}$ or $y_n\in e_a\setminus V$ for some $a\in\mathcal{A}$. In the first case, the $0$th and $(-1)$st tiles of $\phi_{s[0,n-1]}^{-1}(\mathcal{T},t)$ are determined, and the position of the origin within one of these tiles. In the second case, the $0$th tile is determined and the position of the origin in this tile. Following the above argument this allows us to conclude that arbitrarily large patches around the origin are determined and these patches eventually cover the entire real line, hence $(\mathcal{T},t)$ is fully determined. So, the map $p$ is injective and so also bijective. By usual compactness arguments, we conclude that $p$ is a homeomorphism.
		\end{proof}
		 Let $l(w)$ and $r(w)$ respectively be the leftmost and rightmost letters of the word $w$.
		\begin{definition}
		We define a cellular map $g_i\colon K_{i+1} \to K_i$ between consecutive BD complexes on tile edges by, if $\phi_{s_i}(a)=a_1a_2\ldots a_k$, $$g_i(e_a):=e_{a_1} \cup e_{a_1a_2} \cup e_{a_2} \cup \cdots \cup e_{a_{k-1}a_k} \cup e_{a_k}$$ in an orientation preserving, and uniformly expanding way, and on vertex edges by, if $r(\phi_{s_i}(a))=a_k$ and $l(\phi_{s_i}(b))=b_1$, $$g_i(e_{ab}):=e_{a_kb_1}.$$
		\end{definition}
		The maps $f_i$ and $g_i$ are homotopic. The maps $g_i$ also satisfy the property that $g_i(S_{i+1})\subset S_i$ and $g_i|_{S_{i+1}}$ is simplicial.
		\begin{thm}\label{isomorphism}
		There is an isomorphism of groups
		$$\check{H}^1(\Omega) \cong \lim_{\longrightarrow}(H^1(K_i),g_i^*)$$ between the first \Cech cohomology of the mixed substitution tiling space and the direct limit of induced maps $g_i^*$ acting on the first cohomology groups of the Barge-Diamond complexes.
		\end{thm}
		\begin{proof}
		Since $f_i$ and $g_i$ are homotopic, \Cech cohomology is isomorphic to singular cohomology on CW-complexes, \Cech cohomology is a continuous functor, and $\Omega$ is homeomorphic to $\displaystyle{\lim_{\longleftarrow}(K_i, f_i)}$, we get $$\lim_{\longrightarrow}(H^1(K_i), g_i^*) = \lim_{\longrightarrow}(H^1(K_i), f_i^*) \cong \check{H}^1(\lim_{\longleftarrow}((K_i), f_i)) \cong \check{H}^1(\Omega).$$
		\end{proof}
		\begin{thm}\label{LES}
		Let $|\mathcal{A}|=l$ and $\displaystyle{\Xi=\lim_{\longleftarrow}(S_i, g_i)}$. There is an exact sequence $$0 \to \tilde{H}^0(\Xi) \to \lim_{\longrightarrow}(\mathbb{Z}^l,M_{s_i}^T) \to \check{H}^1(\Omega) \to \check{H}^1(\Xi) \to 0.$$
		\end{thm}
		\begin{proof}
		Consider the sequence of pairs $(K_i,S_i)_{i\geq 0}$. To each pair there is associated a long exact sequence in reduced singular cohomology		
		$$\cdots \to H^{n-1}(S_i)\to H^n(K_i,S_i) \to H^n(K_i) \to H^n(S_i) \to H^{n+1}(K_i,S_i) \to \cdots$$
		which is trivial outside of degree $0$ and $1$. Moreover the spaces $K_i$ and  $K_i/S_i$ are connected, so $\tilde{H}^0(K_i)=0$ and $\tilde{H}^0(K_i,S_i)=0$. There is a commutative diagram		
		$$\begin{tikzpicture}[node distance=3cm, auto]
  		\node (00) {$\tilde{H}^0(S_i)$};
  		\node (10) [node distance=2cm, above of=00] {$\tilde{H}^0(S_{i+1})$};
  		\node (01) [right of=00] {$H^1(K_i,S_i)$};
  		\node (11) [right of=10] {$H^1(K_{i+1},S_{i+1})$};
  		\node (02) [right of=01] {$H^1(K_i)$};
		\node (12) [right of=11] {$H^1(K_{i+1})$};
  		\node (03) [right of=02] {$H^1(S_i)$};
  		\node (13) [right of=12] {$H^1(S_{i+1})$};
  		\node (04) [node distance=2.5cm, right of=03] {$0$};
  		\node (14) [node distance=2.5cm, right of=13] {$0$};
  		\node (0-1)[node distance=2.5cm, left of=00] {$0$};
  		\node (1-1)[node distance=2.5cm, left of=10] {$0$};
  		\draw[->] (00) to node [swap] {$g_i^*$} (10);
  		\draw[->] (01) to node [swap] {$g_i^*$} (11);
  		\draw[->] (02) to node [swap] {$g_i^*$} (12);
  		\draw[->] (03) to node [swap] {$g_i^*$} (13);
  		\draw[->] (00) to node [swap] {} (01);
  		\draw[->] (01) to node [swap] {} (02);
  		\draw[->] (02) to node [swap] {} (03);
  		\draw[->] (03) to node [swap] {} (04);
  		\draw[->] (10) to node [swap] {} (11);
  		\draw[->] (11) to node [swap] {} (12);
  		\draw[->] (12) to node [swap] {} (13);
  		\draw[->] (13) to node [swap] {} (14);
  		\draw[->] (0-1) to node [swap] {} (00);
  		\draw[->] (1-1) to node [swap] {} (10);
		\end{tikzpicture}$$
		whose rows are the exact sequences of the pairs $(K_i,S_i)$, and with vertical homomorphisms induced by the the maps $g_i$. Taking the direct limit along each column of this diagram produces an exact sequence
		$$0 \to \lim_{\longrightarrow}(\tilde{H}^0(S_i),g_i^*) \to \lim_{\longrightarrow}(H^1(K_i,S_i),g_i^*) \to \lim_{\longrightarrow}(H^1(K_i),g_i^*) \to \lim_{\longrightarrow}(H^1(S_i),g_i^*) \to 0.$$		
		 Note that $S_i$ is a closed subcomplex of the CW complex $K_i$, so $(K_i,S_i)$ is a good pair, and we can identify the relative cohomology group $H^n(K_i,S_i)$ with the cohomology of the quotient $H^n(K_i/S_i)$. For $n=1$ this is the first cohomology of a wedge of $l$ circles giving $H^1(K_i,S_i) \cong \mathbb{Z}^l$. Moreover the induced map $g_i^*$ on the relative cohomology acts as the transpose of the substitution matrix $M_{s_i}^T$ on the direct sum of $l$ copies of the integers.
		 
		 Putting this together with the fact that \Cech cohomology is isomorphic to singular cohomology for CW complexes, \Cech cohomology is continuous, and Theorem \ref{isomorphism} completes the proof.
		\end{proof}
		\begin{rem}
		We can say more than this if there is an appropriate notion of an eventual range, as there is in the classical case of a single substitution when $F=\{\phi\}$. In this case, $g_i^{k+1}(S_i)\subset g_i^{k}(S_i)$ for all $k\geq 0$ and as $g_i|_{S_i}$ is simplicial, and $S_i$ is a finite simplicial complex, this intersection must stabilise to some eventual range $S_{ER}=\bigcap_k g_i^k(S_i)$. Since $g_i$ restricts to a simplicial map on $S_{ER}$, the inverse limit of the maps $g_i|_{S_{i+1}}$ is just the inverse limit of a map which permutes simplices in the eventual range. The inverse limit of a sequence of homeomorphisms is homeomorphic to any space appearing in that limit, so $\displaystyle{\lim_{\longleftarrow}}(S_i,g_i)$ is homeomorphic to the eventual range $S_{ER}$, and then the cohomology of this inverse limit is readily determined as the simplicial cohomology of the eventual range.
		
		In our case where $(F,s)$ may not be a trivial system with $s$ constant, we can still determine the inverse limit of $S_i$ under the maps $g_i$ using a similar simplicial analysis, but an eventual range does not always exist.
		\end{rem}
		\begin{lem}\label{simplicial}
		Let $C$ be a finite simplicial complex of dimension $n$ and let $f_i\colon C_{i+1} \to C_i$ be a sequence of simplicial maps between a family $\{C_i\}_{i\geq 0}$ of subcomplexes of $C$. The inverse limit space $\displaystyle{\lim_{\longleftarrow}(C_i,f_i)}$ has the homeomorphism type of a finite simplicial complex of dimension at most $n$. \qed
		\end{lem}
		\begin{cor}\label{cor1}
		Let $m$ be the rank of the first singular cohomology of $\Xi$. If $\Xi$ is connected, then $$\check{H}^1(\Omega) \cong \displaystyle{\lim_{\longrightarrow}}(\mathbb{Z}^l,M_{s_i}^T) \oplus \mathbb{Z}^m.$$
		\end{cor}
		\begin{proof}
		Each of the $S_i$ is a subcomplex of the complex $S_F=\left( \bigcup_{ab\in\mathcal{A}^2} e_{ab}\right)/{\sim}$ which is the complex built from all possible vertex edges $e_{ab}$ with $ab\in\mathcal{A}^2$. The map $g_i$ restricted to $S_{i+1}$ is a simplicial map for each $i$, and so the inverse limit $\Xi$ is a one-dimensional simplicial complex by Lemma \ref{simplicial}. The \Cech cohomology of a simplicial complex is isomorphic to the singular cohomology, giving $\check{H}^1(\Xi) \cong \mathbb{Z}^m$. If  $\Xi$ is connected, then $\tilde{H}^0(\Xi)=0$, and since $\mathbb{Z}^m$ is free abelian, hence projective, the short exact sequence of Theorem \ref{LES} splits to give the result.
		\end{proof}
		If $\Xi$ is not connected then the direct limit of transpose matrices must be quotiented by $k-1$ copies of the integers $\mathbb{Z}$ where $k$ is the number of connected components in $\Xi$. The exact sequence of Theorem \ref{LES} tells you how this group sits inside the direct limit.
		\begin{example}\label{Ex:ArnouxRauzy}
		As a basic example, we can use this result to provide a short proof that $\check{H}^1(\Omega) \cong \mathbb{Z}^d$ for a tiling space $\Omega$ associated to an \emph{Arnoux-Rauzy} sequence on $d$-letters. The Arnoux-Rauzy sequences are a special class of sequences, introduced by Arnoux and Rauzy in \cite{ArnouxRauzy}, belonging to the family of episturmian sequences\footnote{The set of $d$-episturmian sequences is the unique set of sequences on $d$ letters which are closed under reversal and have at most one right special factor of length $k$ for each $k\geq 1$.}, which generalise Sturmian sequences (the case where $d=2$). For an alphabet $\mathcal{A}=\{a_1,\ldots, a_d\}$ on $d$ letters, the Arnoux-Rauzy substitutions are given by the $d$ substitutions
		$$\mu_i:\left\{
			\begin{array}{rcl}
			a_i & \mapsto & a_i \\
			a_j & \mapsto & a_ja_i,\quad j\neq i
			\end{array}
			\right.
		$$
		for each $i\in\{1,\ldots, d\}$. The Arnoux-Rauzy sequences are then the sequences that are admitted by the mixed substitution systems $(F,s)$ where $F=\{\mu_1,\ldots,\mu_d\}$ and $s$ contains an infinite number of terms of each type (to enforce primitivity). So $(F,s)$ is primitive, even though the individual substitutions are not, and is well known to have the unique composition property.
		
		If we consider the map $g_{\mu_i}$, for fixed $i$, acting on the subcomplex of vertex edges (at any level in the inverse limit), the image will always be a subset of the union of the edges $\bigcup_{j\in\{1,\ldots, d\}} e_{ij}$. We can see this by noting that $g_{\mu_i}(e_{kj}) = e_{ij}$ for any $k$. This subcomplex is contractible, hence the image of $g_{\mu_i}|_{S_i}$ is contractible, and so we can conclude that $\Xi$ is also contractible. By Corollary \ref{cor1}, it follows that $\check{H}^1(\Omega_{F,s})$ is isomorphic to $\displaystyle{\lim_{\longrightarrow}(\mathbb{Z}^d, M^T_{\mu_i})}$, but note that each of the transition matrices is invertible over the integers $\mathbb{Z}$ and so actually the direct limit of matrices is just isomorphic to $\mathbb{Z}^d$. This completes the proof.
		\end{example}
		In \cite{GahlerMaloney}, it was shown that the maximum rank of $\check{H}^1$ for a mixed substitution on $d$ letters is $d^2-d+1$. We show that the same result can be reached with a basic combinatorial argument using BD complexes. Let $\operatorname{rk}(G)$ be the rank of the group $G$.
		\begin{prop}[G\"{a}hler-Maloney]
		For a primitive, recognisable mixed substitution system $(F,s)$ on an alphabet $|\mathcal{A}|=d$, the rank of the \Cech cohomology group, $\operatorname{rk}(\check{H}^1(\Omega_{F,s}))$ is bounded above by $d^2-d+1$.
		\end{prop}
		\begin{proof}
		The BD complexes $K_i$ of $(F,s)$ have $2d$ vertices (given by each end of the tile edges $e_a$ for $a\in\mathcal{A}$), and at most $d^2+d$ edges ($d^2$ vertex edges of the form $e_{ab}$ for $a,b\in\mathcal{A}$ and $d$ tile edges $e_a$). So the Euler characteristic is bounded below by $V-E=2d-(d^2+d)=d-d^2$, giving $\chi\geq d-d^2$. By definition of the Euler characteristic, we have $\chi=\operatorname{rk}(H^0(K_i))-\operatorname{rk}(H^1(K_i))$, and so as $K_i$ is necessarily connected, $\operatorname{rk}(H^1(K_i))\leq d^2 - d + 1$. Taking the direct limit of $g_i^*$ acting on $H^1(K_i)$ then tells us that $\operatorname{rk}(\check{H}^1(\Omega))\leq d^2 - d + 1$, because the rank of the limit cannot exceed the bound of the ranks of the approximants.
		\end{proof}
%
		G\"{a}hler and Maloney also showed in \cite{GahlerMaloney}, with a family of examples, that this bound is tight.

		\subsection{The Universal Barge-Diamond Complex}
		A problem to contend with in the construction of the BD complex for a mixed substitution system $(F,s)$ is that there are a potentially large set of complexes and maps which can appear in the inverse limit. It may be helpful for theoretical (though likely not computational) reasons to instead build an inverse system whose approximants are all the same, and where the family of maps appearing in the system are only as large as the family of substitutions $F$. This was achieved for the Anderson-Putnam complexes which appear in \cite{GahlerMaloney} where the so-called \emph{universal Anderson-Putnam complex} was introduced. We are also able to achieve this depending on a compatibility condition of the sequence $s$ of substitutions in the system. This `self-correcting' condition is similar to the one introduced in \cite{GahlerMaloney}.
		
		First, let us define the complex that will be our candidate universal BD complex.
		\begin{definition}
		Let $\mathcal{A}$ be a finite alphabet and let $F=\{\phi_0, \phi_1, \ldots, \phi_k\}$ be a set of substitutions on $\mathcal{A}$. Let $\displaystyle{\epsilon = \min_{a\in\mathcal{A}, \phi\in F} \left\{ \frac{1}{2|\phi(a)|} \right\}}$ be a small positive real number. For $a\in\mathcal{A}$, let
		$$e_a=[\epsilon,1-\epsilon]\times\{a\}$$
		and for $ab\in\mathcal{A}^2$, let
		$$e_{ab}=[-\epsilon,\epsilon] \times \{ab\}.$$
		The \emph{universal Barge-Diamond complex} for $F$ is denoted by $K_{F}$ and is defined to be
		$$K_{F} = \left( \bigcup_{a\in\mathcal{A}} e_a \cup \bigcup_{ab\in\mathcal{A}^2} e_{ab} \right)/{\sim}$$
		where for all $a,b,c\in\mathcal{A}$,
		$$(1-\epsilon,a) \sim (-\epsilon,ab) \sim (-\epsilon, ac) \quad \mbox{and} \quad (\epsilon,a) \sim (\epsilon,ba) \sim (\epsilon, ca).$$
		We also define the \emph{subcomplex of vertex edges} $S_{F}$ of $K_{F}$ by
		 $$S_{F}=\bigcup_{ab\in\mathcal{A}^2} e_{ab}/{\sim}.$$
		\end{definition}
		See Figure \ref{fig:M3} for the universal BD complex for a set of substitutions $F$ on an alphabet on three letters, $\mathcal{A}=\{a,b,c\}$.
		\begin{figure}[h]
		$$
		\begin{tikzpicture}
		\clip (-5,-3) rectangle (5,4.8) ;
		\def \n {6}
		\def \radius {1.5cm}
		\def \margin {3} 

		\foreach \s in {1,...,\n}
		{
  			\node[draw, fill, circle, minimum size=0.15] (node\s) at ({360/\n * (\s - 1)}:\radius) {};
  		}

  			\draw[->, ultra thick] (node1) 
    		to node [right] {$ca$} (node2);
    		\draw[->, ultra thick] (node3) 
    		to node [above] {$aa$} (node2);
    		\draw[->, ultra thick] (node3) 
    		to node [left] {$ab$} (node4);
    		\draw[->, ultra thick] (node5) 
    		to node [left] {$bb$} (node4);
    		\draw[->, ultra thick] (node5) 
    		to node [below] {$bc$} (node6);
    		\draw[->, ultra thick] (node1) 
    		to node [right] {$cc$} (node6);
    		
    		\draw[->, ultra thick, out=60, in=120, looseness=7] (node2) 
    		to node [below] {$a$} (node3);
    		\draw[->, ultra thick, out=180, in=240, looseness=7] (node4) 
    		to node [right] {$b$} (node5);
    		\draw[->, ultra thick, out=300, in=0, looseness=7] (node6) 
    		to node [left] {$c$} (node1);

    		\draw[->, ultra thick, out=40, in=260] (node5)
    		to node [right] {$ba$} (node2);
    		\draw[-, line width=6pt, out=280, in=140, draw=white] (node3) 
    		to (node6);
    		\draw[->, ultra thick, out=280, in=140] (node3) 
    		to node [left] {$ac$} (node6);
    		\draw[-, line width=6pt, draw=white, out=160, in=20] (node1) 
    		to (node4);
    		\draw[->, ultra thick, out=160, in=20] (node1) 
    		to node [above] {$cb$} (node4);
    		\begin{scope}
    		\clip (0,0) rectangle (1,0.5);
    		\draw[-, line width=6pt, draw=white, out=40, in=260] (node5)
    		to (node2);
    		\end{scope}
    		\begin{scope}
    		\clip (0,0) rectangle (1,0.5);
    		\draw[->, ultra thick, out=40, in=260] (node5)
    		to (node2);
    		\end{scope}
    		
    		\node [draw, circle, minimum size=4cm, thick, dashed] (xx) [label=$S_F$]{};

		\end{tikzpicture}
		$$
		\caption{The universal Barge-Diamond complex $K_F$ for $\mathcal{A}=\{a,b,c\}$.}
		\label{fig:M3}
		\end{figure}
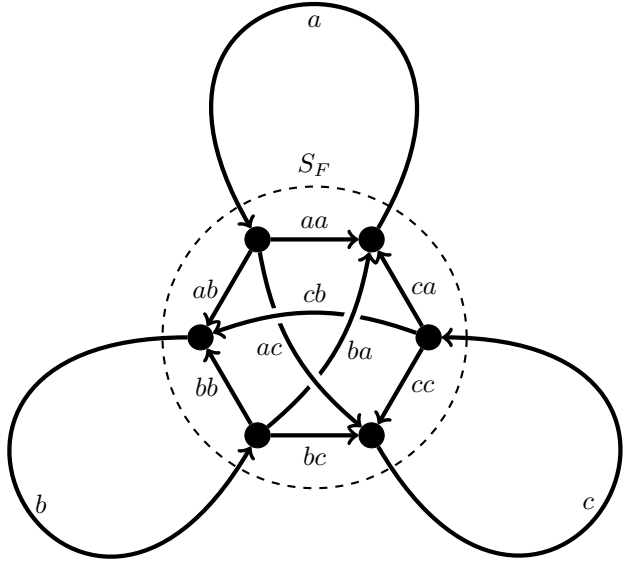
		\begin{definition}
		For each $\phi\in F$, we define a map $g_{\phi} \colon K_{F}\to K_{F}$ on the universal BD complex for $F$ by, if $\phi(a)=a_1a_2 \ldots a_k$, $$g_{\phi}(e_a) := e_{a_1} \cup e_{a_1a_2} \cup\cdots \cup e_{a_{k-1}a_k} \cup e_{a_k}$$ in an orientation preserving and uniformly expanding way, and on vertex edges by, if $r(\phi(a))=a_k$ and $l(\phi(b))=b_1$, $$g_{\phi}(e_{ab}) := e_{a_kb_1}.$$
		\end{definition}
		\begin{definition}
		Let $(F,s)$ be a mixed substitution system on the alphabet $\mathcal{A}$. Let $\mathcal{A}^2$ be the set of two-letter words in symbols from $\mathcal{A}$. If for every $i\geq 0$, there exists an $N\geq 1$ such that for all $ab\in\mathcal{A}^2$ there is $cd\in\mathcal{L}^2_{F,\sigma^{i}(s)}$ such that $$r(\phi_{s[i,i+N]}(a))l(\phi_{s[i,i+N]}(b))=cd,$$  then we say $(F,s)$ is \emph{self-correcting}.
		\end{definition}
		\begin{rem}
		The self-correcting property has been introduced for similar reasons as to why the property was introduced in \cite{GahlerMaloney}. Note that the definitions are not the same, and our definition is tailored to work specifically in the Barge-Diamond setting. Self-correcting substitution systems are sufficient to allow us to add cells to the complexes appearing in the inverse limit representation of the tiling space of the system, and not change the cohomology of the inverse limit. In particular, we can use the same universal BD complex at each level of the inverse system if $(F,s)$ is self correcting.
		\end{rem}
		\begin{example}
		Let $\mathcal{A}=\{a,b\}$ be an alphabet on two letters and let $F=\{\phi_0\}$ be the single substitution given by
		$$
		\phi_0:\left\{
			\begin{array}{rcl}
			a & \mapsto & b \\
			b & \mapsto & ba
			\end{array}
			\right. ,
		$$
		the Fibonacci substitution. There is only one possible sequence of substitutions to consider which is the constant sequence $s=(0,0,\ldots)$. We see that
		$$
		\begin{array}{rcl}
		r(\phi_0(a))l(\phi_0(a)) & = & bb\\
		r(\phi_0(a))l(\phi_0(b)) & = & bb
		\end{array}
		\begin{array}{rcl}
		r(\phi_0(b))l(\phi_0(a)) & = & ab\\
		r(\phi_0(b))l(\phi_0(b)) & = & ab
		\end{array}
		$$
		which are all admitted two-letter strings in $\mathcal{L}^2_{\phi_0}=\mathcal{L}^2_{F,\sigma^i(s)}=\{ab,ba,bb\}$ for all $i\geq 0$, and so $(F,s)$ is self-correcting.
		\end{example}
		\begin{example}
		The non-degenerate mixed Chacon substitution systems which appear in Section \ref{chacon} are all automatically self-correcting because their set of admitted two-letter words is complete. That is, $\mathcal{L}^2_{F,\sigma^i(s)}=\mathcal{A}^2$ for all $i\geq 0$.
		\end{example}
		\begin{example}
		The mixed substitution system associated to an Arnoux-Rauzy sequence on an alphabet with $d$ letters, as introduced in Example \ref{Ex:ArnouxRauzy}, can be seen to be self-correcting. Let $m\geq 0$ be fixed and suppose $s_m=i$. We have that $r(\mu_i(a_ja_k))l(\mu_i(a_ja_k))=a_ia_j$ which is clearly admitted by $(F,\sigma^m(s))$ because by primitivity, $a_j$ appears somewhere in the sequences appearing in $\Sigma_{F,\sigma^{m+1}}$, and preceded by some letter $a_{j'}$. The substituted word $\mu_i(a_{j'} a_j)$ is $a_i a_j a_i$ if $j'=i$, or $a_{j'} a_i a_j a_i$ if $j'\neq i$, both of which contain $a_i a_j$. So all two-letter words are corrected by $(F,\sigma^m(s))$ after one substitution for all $m\geq 0$, hence $(F,s)$ is self correcting.
		\end{example}
		\begin{example}
		Let $F=\{\phi_0\}$ be the single substitution given by
		$$
		\phi_0:\left\{
			\begin{array}{rcl}
			a & \mapsto & aaba \\
			b & \mapsto & bab
			\end{array}
			\right. .
		$$
		Again, there is only one possible sequence of substitutions to consider which is the constant sequence $s=(0,0,\ldots)$. This time, we see that
		$$
		\begin{array}{rcl}
		r(\phi_0(a))l(\phi_0(a)) & = & aa\\
		r(\phi_0(a))l(\phi_0(b)) & = & ab
		\end{array}
		\begin{array}{rcl}
		r(\phi_0(b))l(\phi_0(a)) & = & ba\\
		r(\phi_0(b))l(\phi_0(b)) & = & bb
		\end{array}
		$$
		but the only admitted two-letter strings for $\phi_0$ are $\mathcal{L}^2_{\phi_0}=\mathcal{L}^2_{F,\sigma^i(s)}=\{aa,ab,ba\}$ for all $i\geq 0$. Since the transition of $bb$ is fixed under substitution, this pair will never be corrected to an admitted pair, and so $(F,s)$ is \textbf{not} self-correcting because $bb$ is not admitted by $\{\phi_0\}$.
		\end{example}
		\begin{thm}\label{Universalhomeomorphism}
		For a self-correcting mixed substitution system $(F,s)$, there is an equality $$\lim_{\longleftarrow}(K_i, g_{s_i}) = \lim_{\longleftarrow}(K_{F}, g_{\phi_{s_i}})$$ of the inverse limits of induced cellular substitution maps on the usual Barge-Diamond complexes $K_i$ and on the universal Barge-Diamond complex $K_F$, seen as subsets of $\prod_{i\geq 0} K_F$.
		\end{thm}
		\begin{proof}
		It is clear that $\displaystyle{\lim_{\longleftarrow}(K_i, g_{\phi_{s_i}})}$ is a subset of $\displaystyle{\lim_{\longleftarrow}(K_{F},g_{\phi_{s_i}})}$ because for every $i$, $K_i \subset K_F$. In order to show the other inclusion, we make use of the self-correcting property of $(F,s)$.
		
		Pick a point $x=(x_0,x_1,x_2,\ldots) \in \displaystyle{\lim_{\longleftarrow}(K_{F},g_{\phi_{s_i}})}$. If $x_i$ is a point in $K_i$ for every $i\geq 0$, then $x\in \displaystyle{\lim_{\longleftarrow}(K_i, g_{\phi_{s_i}})}$ and we are done. If there is some $i\geq 0$ such that $x_i\notin K_i$ then $x_i$ must be in the interior of a vertex edge $e_{ab}$ of $S_F$ which is not a vertex edge appearing in $S_i$. Let $N\geq 1$ and be such that for all $ab\in\mathcal{A}^2$, there are $c,d\in\mathcal{A}$ such that $$r(\phi_{s[i,i+N]}(a))l(\phi_{s[i,i+N]}(b))=cd$$ and with $cd\in\mathcal{L}_{F,\sigma^{i+N}(s)}$. The integer $N$ exists because $(F,s)$ is self-correcting. Note that for all $z\in S_F$, we have $g_{\phi_{s[i,i+N]}}(z)\in S_i$ and so there exists no $x_{i+N}$ such that $g_{\phi_{s[i,i+N]}}(x_{i+N})=x_i$. From the definition of the inverse limit then, such an $x=(x_0,x_1,x_2,\ldots)$ with some $x_i$ in $K_F\setminus K_i$ cannot exist. It follows that for all $i\geq 0$, we must have $x_i\in K_i$ and so $x \in \displaystyle{\lim_{\longleftarrow}(K_i, g_{\phi_{s_i}})}$. Hence $\displaystyle{\lim_{\longleftarrow}(K_i, g_{\phi_{s_i}}) = \lim_{\longleftarrow}(K_{F},g_{\phi_{s_i}})}$.
		\end{proof}		
		\begin{cor}
		For a primitive, recognisable, self-correcting mixed substitution system $(F,s)$, there is an isomorphism $$\check{H}^1(\Omega) \cong \lim_{\longrightarrow}(H^1(K_{F}),g_{\phi_{s_i}}^*)$$ between the \Cech cohomology of the mixed substitution tiling space and the direct limit of the induced homomorphisms on the cohomology of the universal Barge-Diamond complex for $F$.
		\end{cor}
		\begin{proof}
		Making use of Theorem \ref{isomorphism} it suffices to show that there is an isomorphism of direct limits $$\displaystyle{\lim_{\longrightarrow}(H^1(K_{F}),(g_{\phi_{s_i}})^*)} \cong \displaystyle{\lim_{\longrightarrow}(H^1(K_i),(g_{\phi_{s_i}})^*)}$$ which follows from Theorem \ref{Universalhomeomorphism} and the continuity and functorality of \Cech cohomology.
		\end{proof}

%
%
%
%
%

	\section{Uncountability of Set of Cohomology Groups}\label{chacon}
	The machinery is now in place to be able to introduce our new example and prove the main result, Theorem \ref{MainTheorem}.
		\subsection{The Mixed Chacon Substitution System}
		Let $\mathcal{A}=\{a,b\}$ and let $F=\{\psi_0, \psi_1, \psi_2\}$ be the set of substitutions $\psi_i\colon \mathcal{A}^* \to \mathcal{A}^*$ given by
		$$\begin{array}{ccc}
			\psi_0:\left\{
			\begin{array}{rcl}
			a & \mapsto & aabba \\
			b & \mapsto & b
			\end{array}
			\right., 
			&
			
			\psi_1:\left\{
			\begin{array}{rcl}
			a & \mapsto & aab \\
			b & \mapsto & bba 
			\end{array}
			\right.,
			&
			
			\psi_2:\left\{
			\begin{array}{rcl}
			a & \mapsto & a \\
			b & \mapsto & bbaab 
			\end{array}
			\right.
		\end{array}.$$
		We call $(F,s)$ a \emph{mixed Chacon substitution system} because $\psi_0$ and $\psi_2$ are each mutually locally derivable to the classical Chacon substitution. The substitution $\psi_1$ is not strictly necessary to achieve the final result\footnote{In fact we only need any two of the three in order to still be able to use the uncountability result of Goodearl and Rushing.}, but it seems natural to include $\psi_1$ in the system for aesthetic reasons.
		
		There are associated substitution matrices $$M_0=\begin{pmatrix}3 & 0\\ 2 & 1\end{pmatrix},\: M_1=\begin{pmatrix}2 & 1\\ 1 & 2\end{pmatrix},\: M_2=\begin{pmatrix}1 & 2\\ 0 & 3\end{pmatrix}.$$		
		Let $L=\begin{pmatrix}0 & 1\\ 1 & 1\end{pmatrix}$ which has inverse given by $L^{-1}=\begin{pmatrix}-1 & 1\\ 1 & 0\end{pmatrix}$ and note the identities
		\begin{eqnarray}
			\label{matrices1}LM_0^T L^{-1} & = \begin{pmatrix}1 & 0\\ 0 & 3\end{pmatrix} = B_0 \\
			\label{matrices2}LM_1^T L^{-1} & = \begin{pmatrix}1 & 1\\ 0 & 3\end{pmatrix} = B_1 \\
			\label{matrices3}LM_2^T L^{-1} & = \begin{pmatrix}1 & 2\\ 0 & 3\end{pmatrix} = B_2
		\end{eqnarray}
		
		Let $\alpha\in\mathbb{Z}_3$ be a $3$-adic integer with digits $\ldots\epsilon_2 \epsilon_1 \epsilon_0$. We only wish to consider a specific family of such $3$-adic integers. Let $s_{\alpha}=(s_0, s_1, s_2, \ldots)$ be the associated sequence of digits appearing in $\alpha$, so $s_n=\epsilon_n$.
		\begin{definition}
		A sequence $s= (s_0, s_1, s_2,\ldots)\in\{0,1,2\}^{\mathbb{N}}$ is \emph{degenerate} if there exists a natural number $N$ such that either $s_n=0$ for all $n\geq N$ or $s_n=2$ for all $n\geq N$. That is, the sequence $s$ is eventually constant $0$s or constant $2$s. We say a $3$-adic integer $\alpha$ is \emph{degenerate} if its associated sequence of digits $s_{\alpha}$ is degenerate.
		\end{definition}
		\begin{rem}
		Non-degeneracy is only a technical condition which forces weak primitivity of the relevant mixed substitution sequences, allowing us to use the previous results on primitive sequences of substitutions. A similar condition can also be defined to enforce strong primitivity, but in this case it is not necessary. Removing these spurious cases helps to simplify the proof\footnote{In fact this only removes, up to homeomorphism, a single tiling space from the family of spaces we are studying -- the one associated to a constant sequence of substitutions $s=0,0,0,\ldots$, which is the usual Chacon tiling space and is well studied.}.
		\end{rem}
		\begin{prop}\label{degenerate}
		Let $F=\{\psi_0,\psi_1,\psi_2\}$. For a $3$-adic integer $\alpha \in \mathbb{Z}_3$, the associated system of mixed substitutions $(F, s_{\alpha})$ is weakly primitive if and only if $\alpha$ is non-degenerate.
		\end{prop}
		\begin{proof}
		The necessity of non-degeneracy of $\alpha$ for primitivity of $(F,s_{\alpha})$ is clear, as all positive powers of $M_0$ and $M_2$ contain a zero-entry.
		
		For sufficiency, let $n\geq 0$ be given. It is easy to verify that for $i,j\in\{0,1,2\}$, all products $M_i M_j$ have strictly positive entries except for $M_0 M_0$ and $M_2 M_2$. If $M_{s_n}=M_1$ then we are done. If $M_{s_n}=M_0$ then by the non-degeneracy of $\alpha$, there exists a least $k\geq 1$ such that $M_{s_{n+k}}=M_1$ or $M_2$. In both cases the matrix $M_{s[n+k-1,n+k]}$ has strictly positive entries and so, since the matrix $M_{s[n,n+k-2]}$ has strictly positive diagonal, we conclude that $M_{s[n,n+k]}$ has strictly positive entries. This similarly holds if $M_{s_n}=M_2$. So $(F,s_{\alpha})$ is weakly primitive.
		\end{proof}
		\begin{prop}\label{uncountable}
		There are an uncountable number of non-degenerate $3$-adic integers.
		\end{prop}
		\begin{proof}There are only a countable number of degenerate $3$-adic integers given by those sequences with some finite initial string followed by a constant tail of $0$s or $2$s, since the set of finite initial strings is countable and the union of two countable sets is countable. The complement of a countable subset of an uncountable set is uncountable and so the non-degenerate $3$-adic integers form an uncountable set.
		\end{proof}		
		

		\subsection{Calculating Cohomology for the Mixed Chacon Tilings}
		\begin{lem}\label{recognisable}
		If $\alpha$ is a non-degenerate $3$-adic integer then $(F,s_{\alpha})$ is recognisable.
		\end{lem}
		\begin{proof}
		We use the equivalence between recognisability and the unique composition property. Suppose $s_i=1$, then for any $\mathcal{T}\in\Sigma_{F,\sigma^i(s_{\alpha})}$ we note that the symbol $a$ can only appear in a string of length $1,2$ or $3$. We can partition the symbols appearing in the sequence $\mathcal{T}$ according to the rule that: 
		\begin{itemize}
		\item If $a$ appears in the string $bab$ then we know the patch extends to $b(bab)ba$ and $a$ belongs to a substituted word $\psi_{s_i}(b)\psi_{s_i}(b)$.
		\item If $a$ appears in the string $baab$ then we know the patch extends to $aa(baab)$ and $a$ belongs to the substituted word $\psi_{s_i}(a)\psi_{s_i}(a)$.
		\item If $a$ appears in the string $baaab$ then we know the patch extends to $b(baaab)$ and $a$ belongs to a substituted word $\psi_{s_i}(b)\psi_{s_i}(a)$.
		\end{itemize}
		As the symbol $a$ appears in any substituted word, this is enough to see that unique composition holds in this case.
		
		If $s_i=0$ then we note that the symbol $b$ either appears\footnote{It is not obvious that length $4$ cannot occur. This is true because, of the three substitutions $\psi_i$, only $\psi_0$ can produce a string of $4$ $b$s, but only when a similar string has been produced in an earlier substituted word. By induction this never happens.} in a string of length $1,2$ or $3$. We can partition the symbols appearing in the sequence $\mathcal{T}$ according to the rule that: 
		\begin{itemize}
		\item If $b$ appears in a string $aba$ then we know the patch extends to $aabb(aba)abba$ and $b$ belongs to the substituted word $\psi_{s_i}(a)\psi_{s_i}(b)\psi_{s_i}(a)$.
		\item If $b$ appears in a string $aabba$ then we know the patch extends to $aabba$ and $b$ belongs to the substituted word $\psi_{s_i}(a)$.
		\item If $b$ appears in a string $babba$ then we know the patch extends to $aab(babba)abba$ and $b$ belongs to the substituted word $\psi_{s_i}(a)\psi_{s_i}(b)\psi_{s_i}(b)\psi_{s_i}(a)$.
		\item If $b$ appears in a string $abbba$ then we know the patch extends to $aabb(abbba)abba$ and $b$ belongs to the substituted word $\psi_{s_i}(a)\psi_{s_i}(b)\psi_{s_i}(b)\psi_{s_i}(b)\psi_{s_i}(b)\psi_{s_i}(a)$.		
		\end{itemize}
		As the symbol $b$ appears in any substituted word, this is enough to see that unique composition holds in this case. Without loss of generality, if $s_i=2$ unique composition also holds, based on the above case $s_i=0$ but with the roles of $a$ and $b$ reversed.
		\end{proof}
		We are now in a position to prove our main result. Recall from Section \ref{goodearl}, and the proof of the Goodearl-Rushing result, that for $\alpha = \ldots \epsilon_2 \epsilon_1 \epsilon_0$ a $3$-adic integer, $G_{\alpha}:=\displaystyle{\lim_{\longrightarrow}(B_{\epsilon_n})}$.
		\begin{thm}
		If $\alpha$ is a non-degenerate $3$-adic integer then $$\check{H}^1(\Omega_{F,s_{\alpha}}) \cong G_{\alpha} \oplus \mathbb{Z}.$$
		\end{thm}	
		\begin{proof}
		Let $(F,s_{\alpha})$ be a mixed substitution system for non-degenerate $\alpha\in\mathbb{Z}_3$. By Proposition \ref{degenerate} and Lemma \ref{recognisable}, $(F,s_{\alpha})$ is a primitive, recognisable mixed substitution system and so Theorem \ref{homeomorphism} applies. This means we can calculate the first \Cech cohomology of the space $\Omega_{F,s_{\alpha}}$ using Theorem \ref{LES}. It is easy to show that for every $i$, $$\mathcal{L}^2_{F,\sigma^i(s_{\alpha})}=\{aa,ab,ba,bb\}$$ and so the BD complex for $(F,\sigma^i(s_{\alpha}))$ is given by the two tile edges for the tiles $a$ and $b$ and all four possible vertex edges as shown in Figure \ref{fig:M2}.		
		\begin{figure}[h]
		$$
		\begin{tikzpicture}[node distance=2cm, auto]
		\clip (-5,-2.5) rectangle (5,2.5) ;
		\node [draw, circle, minimum size=3.7cm, thick, dashed] (xx) [label=$S_{F,\sigma^i(s)}$]{};
  		\node [draw, fill, circle, minimum size=.15cm] (00) [below left of=xx, node distance=1.4142cm] {};
  		\node [draw, fill, circle, minimum size=.15cm] (10) [above of=00] {};
  		\node [draw, fill, circle, minimum size=.15cm] (01) [right of=00] {};
  		\node [draw, fill, circle, minimum size=.15cm, use as bounding box] (11) [right of=10] {};
  		
  		
  		\draw [->,ultra thick] (00) to node {$aa$} (10);
  		\draw [->,ultra thick] (11) to node {$bb$} (01);
  		\draw [->,ultra thick] (00) to node [swap] {$ab$} (01);
  		\draw [->,ultra thick] (11) to node [swap] {$ba$} (10);

  		\path (10) edge[ out=140, in=220
                , looseness=0.1, loop
                , distance=5cm, ->
                , ultra thick]
            node {$a$} (00);
            
        \path (01) edge[ out=320, in=40
                , looseness=0.1, loop
                , distance=5cm, ->
                , ultra thick]
            node {$b$} (11);
		\end{tikzpicture}
		$$
		\caption{The Barge-Diamond complex for $(F,\sigma^i(s_{\alpha}))$.}\label{fig:M2}
		\end{figure}
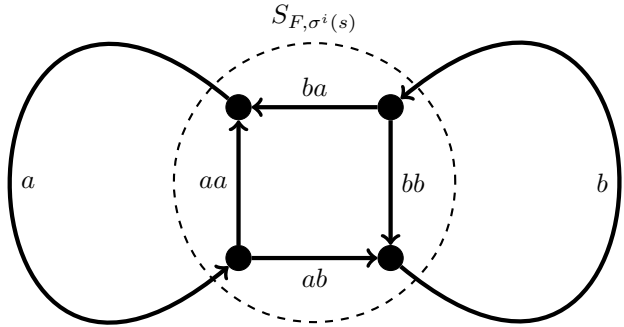
		
We see that the subcomplex $S_{F,\sigma^i(s_{\alpha})}$ is homeomorphic to a circle for all $i$ and the induced maps $g_i$ either fix $S_{F,\sigma^i(s_{\alpha})}$ (if $s_i=0,2$) or reflect $S_{F,\sigma^i(s_{\alpha})}$ (if $s_i=1$). So $\Xi$ is topologically a circle and then $\tilde{H}^0(\Xi)=0$ and $\check{H}^1(\Xi)=\mathbb{Z}$. The exact sequence from \ref{LES} then becomes a split short exact sequence $$0\to \lim_{\longrightarrow}(\mathbb{Z}^2,M_{\epsilon_n}^T) \to \check{H}^1(\Omega_{F,s_{\alpha}}) \to \mathbb{Z} \to 0.$$
		So the first \Cech cohomology $\check{H}^1(\Omega_{F,s_{\alpha}})$ of the tiling space is given by $$\check{H}^1(\Omega_{F,s_{\alpha}})\cong \lim_{\longrightarrow}(M_{\epsilon_n}^T) \oplus \mathbb{Z}.$$
Using equations $(\ref{matrices1})$ -- $(\ref{matrices3})$ we see that the diagram
		$$\begin{tikzpicture}[node distance=2cm, auto]
  		\node (00) {$\mathbb{Z}^2$};
  		\node (10) [above of=00] {$\mathbb{Z}^2$};
  		\node (01) [right of=00] {$\mathbb{Z}^2$};
  		\node (11) [right of=10] {$\mathbb{Z}^2$};
  		\node (02) [right of=01] {$\mathbb{Z}^2$};
  		\node (12) [right of=11] {$\mathbb{Z}^2$};
  		\node (03) [right of=02] {$\cdots$};
  		\node (13) [right of=12] {$\cdots$};
  		\draw[->] (10) to node [swap] {$\stackrel{L}{\cong}$} (00);
  		\draw[->] (11) to node [swap] {$\stackrel{L}{\cong}$} (01);
  		\draw[->] (12) to node [swap] {$\stackrel{L}{\cong}$} (02);
  		\draw[->] (00) to node [swap] {$B_{\epsilon_0}$} (01);
  		\draw[->] (01) to node [swap] {$B_{\epsilon_1}$} (02);
  		\draw[->] (02) to node [swap] {$B_{\epsilon_2}$} (03);
  		\draw[->] (10) to node {$M^T_{\epsilon_0}$} (11);
  		\draw[->] (11) to node {$M^T_{\epsilon_1}$} (12);
  		\draw[->] (12) to node {$M^T_{\epsilon_2}$} (13);  
		\end{tikzpicture}$$	
commutes, giving $\displaystyle{\lim_{\longrightarrow}(M_{\epsilon_n}^T) \cong \lim_{\longrightarrow} (B_{\epsilon_n})}$, which by definition is $G_{\alpha}$. We conclude that $$\check{H}^1(\Omega_{F,s_{\alpha}}) \cong \lim_{\longrightarrow} (M_{\epsilon_n}^T) \oplus \mathbb{Z} \cong G_{\alpha} \oplus \mathbb{Z}.$$
		\end{proof}
		\begin{cor}[Theorem \ref{MainTheorem}]
		There exists a family of minimal mixed substitution tiling spaces exhibiting an uncountable collection of distinct isomorphism classes of first \Cech cohomology groups.
		\end{cor}
		\begin{proof}
		It is a general result that two abelian groups $A$ and $B$ are isomorphic if and only if $A\oplus\mathbb{Z}$ and $B\oplus\mathbb{Z}$ are isomorphic. So if $G_{\alpha}\oplus\mathbb{Z}\cong G_{\alpha'}\oplus\mathbb{Z}$ then $G_{\alpha}\cong G_{\alpha'}$. By the definition of $\sim$-equivalent $3$-adic integers then, we see that for $\alpha,\alpha'\in\mathbb{Z}_3$, the groups $\check{H}^1(\Omega_{F,s_{\alpha}})$ and $\check{H}^1(\Omega_{F,s_{\alpha'}})$ are isomorphic if and only if $G_{\alpha}\oplus\mathbb{Z}$ and $G_{\alpha'}\oplus\mathbb{Z}$ are isomorphic, which is if and only if $\alpha \sim \alpha'$.
		
		By Theorem \ref{countable} these equivalence classes are all countable. Recall that (assuming the axiom of choice) a countable disjoint union of countable sets is countable. As there are an uncountable number of non-degenerate $3$-adic integers by Proposition \ref{uncountable}, and the $\sim$-equivalence classes partition this set into countable subsets, it follows that there are an uncountable number of distinct isomorphism classes of first \Cech cohomology groups $\check{H}^1(\Omega_{F,s_{\alpha}})$ for non-degenerate $3$-adic integers $\alpha$.
		\end{proof}
		We remark that as a consequence, for \emph{most} $3$-adic integers $\alpha$, the cohomology groups $\check{H}^1(\Omega_{F,s_{\alpha}})$ cannot be written in the form
		\begin{equation}\label{pathalogical}
		A \oplus \left(\mathbb{Z}\left[\frac{1}{n_1}\right] \oplus \cdots \oplus \mathbb{Z}\left[\frac{1}{n_k}\right]\right)
		\end{equation}
		for finitely generated abelian group $A$ and natural numbers $n_i$, $1\leq i \leq k$. Of interest is that these cohomology groups all appear as embedded subgroups of $\mathbb{Q}^2\oplus\mathbb{Z}$, and project onto a proper and full rank subgroup of the first summand. In particular they all have rank $3$ as abelian groups, so the number of different isomorphism classes of cohomology is being governed by the complexity of the subgroup structure of the group $\mathbb{Q}^2$.

	If we call groups which cannot be written in the above form (\ref{pathalogical}) \emph{pathological}, it is interesting to ask how typical it is that such pathological cohomology groups appear for general tilings. Does there exist a single substitution with associated cohomology of pathological type? The author was originally motivated by this problem and it appears to still be open.





\end{document}